\documentclass[12pt, twoside, reqno, psamsfonts]{amsart}
\usepackage[english]{babel}
\usepackage[T1]{fontenc}
\usepackage{graphicx}
\usepackage{tikz}
\usepackage{subfigure, float}
\usepackage[left=3.5cm,top=3cm,right=3.5cm,bottom=2.5cm]{geometry}
\usepackage[bf, small, hang]{caption}
\usepackage{microtype}
\makeatletter
\renewcommand\section{\@startsection{section}{1}{\z@}%
    {.7\linespacing\@plus\linespacing}%
    {.5\linespacing}%
    {\centering \scshape  \bf}}
\makeatother
\theoremstyle{plain}                            
\newtheorem{lemma}{Lemma}[section]
\newtheorem{theorem}[lemma]{Theorem}
\newtheorem{proposition}[lemma]{Proposition}
\newtheorem{corollary}[lemma]{Corollary}
\theoremstyle{definition}                        
\newtheorem{definition}[lemma]{Definition}
\newtheorem{example}[lemma]{Example}
\newtheorem*{assumption1}{Assumption A1}
\newtheorem*{assumption2}{Assumption A2}

\theoremstyle{remark}            
\newtheorem{remark}[lemma]{Remark}

\begin{document}
\title[Implicit iteration methods in Hilbert scales]
{Implicit iteration methods in Hilbert scales under general
smoothness conditions}
\author{Qinian Jin}
\author{Ulrich Tautenhahn}
\address{Q. Jin,
Department of Mathematics, Virginia Tech Blacksburg, 522 McBryde
Hall, VA 24061-0123, U.S.A.} \email{qnjin@math.vt.edu}
\address{U. Tautenhahn, Department of Mathematics, University of
Applied Sciences Zittau/G\"orlitz, P.O.Box 1455, 02755 Zittau,
Germany} \email{u.tautenhahn@hs-zigr.de}
%
\date{\today}
\keywords{Ill-posed problems, inverse problems, regularization,
Hilbert scales, implicit iteration method, order optimal error
bounds, general smoothness conditions, operator monotone functions}
\subjclass[2000]{47A52, 65F22, 65J20, 65M30}
\begin{abstract}
For solving linear ill-posed problems regularization methods are
required when the right hand side is with some noise. In the present
paper regularized solutions are obtained by implicit iteration
methods in Hilbert scales.
By exploiting operator monotonicity of certain functions and
interpolation techniques in variable Hilbert scales, we study these
methods under general smoothness conditions. Order optimal error
bounds are given in case the regularization parameter is chosen
either {\it a priori} or {\it a posteriori} by the discrepancy
principle. For realizing the discrepancy principle, some fast
algorithm is proposed which is based on Newton's method applied to
some properly transformed equations.
\end{abstract}
\maketitle

%
\section{Introduction}\label{s1}
In this paper we are interested in solving ill-posed problems
\begin{equation}\label{f11}
A x=y,
\end{equation}
where $A  \in {\mathcal L}(X,Y)$ is a linear, injective and bounded
operator with non-closed range ${\mathcal R} (A)$ and $X, Y$ are
Hilbert spaces with corresponding inner products $(\cdot, \cdot)$
and norms $\| \cdot \|$. Throughout we assume that $y \in {\mathcal
R} (A)$ so that (\ref{f11}) has a unique solution $x^\dagger \in X$.
We further assume that $y$ is unknown and $y^\delta \in Y$ is the
available noisy right hand side with
\begin{equation*}
\|y-y^\delta\|\le \delta .
\end{equation*}

In recent literature many aspects of treating ill-posed problems
with noisy right hand side have been studied. For an overview, see,
e.\,g., the textbooks \cite{ehn96,vv86}.
The numerical treatment of ill-posed problems (\ref{f11}) with noisy
data $y^\delta$ requires the application of special regularization
methods. In this paper we study {\it implicit iteration methods in
Hilbert scales}, in which regularized solutions $x_n^\delta$ are
obtained by
\begin{equation}\label{f12}
x_k^\delta =  x_{k-1}^\delta - (A^*A + \alpha_k B^{2s} )^{-1} A^* (
 A x_{k-1}^\delta - y^\delta ), \quad k=1,2,...,n, \quad x_0^\delta =
x_0
\end{equation}
where $B: {\mathcal D} (B) \subset X \to X$ is some unbounded
densely defined self-adjoint strictly positive definite operator,
$\alpha_k > 0$ are properly chosen real numbers, $s$ is some
generally nonnegative number that controls the smoothness to be
introduced into the regularization procedure and $x_0$ is some
properly chosen starting value. In these regularization methods, the
positive number
\begin{equation}\label{f13}
\sigma_n := \sum_{k=1}^n \frac{1}{\alpha_k}
\end{equation}
plays the role of the regularization parameter.
For results on convergence rates of this method in the special case
$s=0$ we recommend the paper \cite{hg98}, and for some extensions to
the nonlinear case we recommend \cite{ha10,ji10} and
\cite{hho09,lr09,pb10}.

Method (\ref{f12}) with $n=1$ and $x_0 =0$ is the method of Tikhonov
regularization in Hilbert scales which has been studied by Natterer
\cite{na84}. From this paper we know that under the assumptions $\|
B^{-a}x\| \sim \|A x\|$ and $\| B^p x^\dagger\| \le E$ the Tikhonov
regularized solution $x_{\alpha}^{\delta}$ of problem (\ref{f12})
with $n=1$, $x_0 = 0$ and $\alpha_1 = \alpha$  guarantees order
optimal error bounds
\begin{equation}\label{f14}
\| x_{\alpha}^{\delta} - x^\dagger\| = O(\delta^{p/(a+p)}) \quad
\mbox{for} \quad p \le 2s+a
\end{equation}
in case $\alpha$ is chosen {\it a priori} by $\alpha \sim
\delta^{2(a+s)/(a+p)}$.
This result has been extended
\begin{enumerate}
\item[(i)] to the case of choosing $\alpha$ {\it a posteriori} by the
discrepancy principle, see, e.\,g., \cite{ne88,st94,ta93,ta96-1},
\item[(ii)] to a general regularization scheme,  see, e.\,g.,
\cite{ehn96,ta96-1},
\item[(iii)] to the case of
general source conditions including infinitely smoothing operators
$A$, see, e.\,g., \cite{he95,ma94,mt06,mt07,npt05},
\item[(iv)] to the case of nonlinear ill-posed problems, see,
e.\,g., \cite{ehn96,hp08,ln97,ne92,st94,ta98-1}.
\end{enumerate}

Our paper is organized as follows.
In Section 2 we collect some preliminaries on properties of the
implicit iteration methods in Hilbert scales, formulate our general
smoothness conditions and give some consequences that follow from
the general smoothness conditions by exploiting either operator
monotonicity or interpolation in variable Hilbert scales. In Section
3 we treat the case of {\it a priori parameter choice} of the
regularization parameter $\sigma_n$ and Section 4 treats the case of
choosing $\sigma_n$ {\it a posteriori} by the discrepancy principle.
In Section 5 we discuss practical issues of choosing the starting
value $x_0$ and the parameters $s$, $n$ and $\alpha_k$, $k=1,...,n$.
In particular, some fast globally convergent algorithm for realizing
the discrepancy principle is proposed which is based Newton's method
applied to some properly transformed equations. For testing the
algorithm, numerical experiments are performed in Section 6.
%
%
\section{Preliminaries}\label{s2}
\setcounter{equation}{0}
\subsection{Properties of the regularization method}
In our further study, instead with $B$, we will work with the
inverse $G=B^{-1}$. Following proposition gives us some equivalent
representation for $x_n^\delta$ defined by (\ref{f12}) along with
some preliminary properties which will be useful for deriving order
optimal error bounds.
\begin{proposition}\label{p21}
Let $T=AG^s$, $G=B^{-1}$ and $\sigma_n$ be defined by (\ref{f13}).
Then, the regularized solution $x_n^\delta$ defined by (\ref{f12})
possesses the representation
\begin{equation}\label{f21}
x_n^\delta - x_0 = G^{s} g_n (T^*T) T^* ( y^\delta - A x_0) \quad
\mbox{with} \quad g_n (\lambda) = \frac{1}{\lambda} \left ( 1 -
\prod_{k=1}^n \frac{\alpha_k}{\lambda + \alpha_k} \right ) .
\end{equation}
In addition, the function $g_n: (0, c] \to (0, \infty)$ with
$c=\|T\|^2$ and the corresponding residual function $r_n (\lambda):=
1 - \lambda g_n (\lambda)$ obey the properties
\begin{eqnarray}
\mbox{\rm (i)}  & &
g_n (\lambda)   \le \sigma_n, \label{f22}\\[1ex]
\mbox{\rm (ii)} & & \lambda g_n (\lambda) \le 1, \label{f23}\\[1ex]
\mbox{\rm (iii)} & & \lambda r_n (\lambda)  \le \sigma_n^{-1} ,
\label{f24}\\[1ex]
\mbox{\rm (iv)} & & r_n (\lambda) \le \sigma_n^{-1} g_n (\lambda) .
\label{f25}
\end{eqnarray}
\end{proposition}
\begin{proof}
The proof of the representation (\ref{f21}) is standard. For the
proof of (i) we follow the paper \cite{hg98} and observe that the
function $r_n (\lambda) = \prod_{k=1}^n \alpha_k / (\lambda +
\alpha_k )$ is monotonically decreasing and convex with $r_n (0) =
1$. From these properties we conclude that
\[
r_n (\lambda) \ge r_n (0) + r_n'(0) \lambda .
\]
Since $r_n (0) = 1$ and $r_n' (0) = - \sigma_n$ we obtain $r_n
(\lambda) \ge 1 - \lambda \sigma_n$, which is equivalent to (i).
The proof of (ii) follows from the representation
$ \lambda g_n (\lambda) = 1 - \prod_{k=1}^n \alpha_k / (\lambda +
\alpha_k ) $.
For the proof of (iv), we multiply (\ref{f25}) by $\lambda \sigma_n
/ r_n (\lambda)$ and obtain the equivalent inequality
\[
\lambda \sum_{k=1}^n \frac{1}{\alpha_k} \le \prod_{k=1}^n \left ( 1
+ \frac{\lambda}{\alpha_k} \right ) - 1.
\]
This inequality, however, always holds true since the left hand side
is the first order term of the polynomial in $\lambda$ on the right
hand side.
For the proof of (iii) we use (iv) and (ii) and obtain $r_n
(\lambda) \le \sigma_n^{-1} g_n (\lambda) \le \sigma_n^{-1} /
\lambda$, which is equivalent to (iii).
\end{proof}

\begin{remark}
Note that in our forthcoming analysis we will also exploit the fact
that $r_n (\lambda) \le 1$ for $\lambda \in (0,c]$ which is a
consequence of the nonnegativity of $g_n (\lambda)$. Further note
that property (iii) of the above proposition tells us that the
regularization method (\ref{f21}) has at least a qualification of
$p_0 = 1$. For the concept of qualification, see \cite{vv86}.
Finally we note that our analysis does not require the full strength
of the properties (i) -- (iii) of Proposition \ref{p21}. Indeed,
property (i) will be exploited for the $\lambda$ -- range $\lambda
\le \sigma_n^{-1}$, and properties (ii) and (iii) will be exploited
for the $\lambda$ -- range $\lambda \ge \sigma_n^{-1}$.
\end{remark}

For deriving order optimal error bounds for $\| x_n^\delta -
x^\dagger\|$ with $x_n^\delta$ defined by (\ref{f21}) we introduce
the regularized solution $x_n$ with exact data by
\[
x_n - x_0 = G^{s} g_n (T^*T) T^* (y - A x_0) .
\]
It can easily be checked that the following error representations
\begin{equation}\label{f26}
x_n^\delta - x_n = G^{s} g_n (T^*T) T^* (y^\delta - y)  \quad
\mbox{and} \quad
x^\dagger - x_n  =  G^{s} r_n (T^*T)  G^{-s} ( x^\dagger - x_0)
\end{equation}
are valid. From these error representations we see that in the
$\|Ax\|$ -- norm and in the $X_p$ -- norm $\| x\|_p:= \| G^{-p} x\|$
we have
\begin{eqnarray}
\| A x_n^\delta - A x_n \| & = & \| T g_n (T^*T) T^* (y^\delta -
y)\| , \label{f27}
\\[1ex]
\| A x^\dagger - A x_n \| & = & \| T  r_n (T^*T)  G^{-s} ( x^\dagger - x_0) \|,  \label{f28} \\[1ex]
\| x_n^\delta - x_n \|_p & = & \| G^{s-p} g_n (T^*T) T^* (y^\delta -
y) \| , \label{f29}
\\[1ex]
\| x^\dagger - x_n \|_p & = & \| G^{s-p} r_n (T^*T)  G^{-s} (
x^\dagger - x_0 ) \| . \label{f210}
\end{eqnarray}
%
\subsection{Smoothness assumptions}
We formulate our smoothness assumptions in terms of some densely
defined unbounded selfadjoint strictly positive operator $B: X \to
X$.
We introduce a {\it Hilbert scale} $(X_r)_{r \in \mathbb R}$ induced
by the operator $B$ which is the completion of $\cap_{k \in \mathbb
R} {\mathcal D} (B^k)$ with respect to the Hilbert space norm
$\|x\|_r = \| B^{r} x \|$, $r \in \mathbb R$.
For technical reasons, instead of $B$ we will work  with the inverse
$G:=B^{-1}$, which is a bounded linear injective and selfadjoint
operator with non-closed range ${\mathcal R}(G)$. Note that the
above Hilbert space norm $\| \cdot \|_r$ may be represented by
\[
\|x\|_r = \| B^r x \| =  \| G^{-r} x \| , \quad r \in \mathbb R  .
\]
In addition, according to \cite{he95,mp03} we call a function
$\varrho: (0, a] \to (0,b]$ an {\it index function} if it is
continuous and strictly increasing with $\varrho(0${\small +}$) = 0$
and assume
\begin{assumption1}
There exist constants $M \ge m > 0$ and some index function
$\varrho: (0, a] \to (0,b]$ with $a=\|G\|$ and $b=\varrho (a)$ such
that
\begin{eqnarray*}
\mbox{(i)} &  m \left \| \varrho (G) \, x \right \| \le \|Ax\|
& \quad \text{for all $x \in X$} ,  \\[0.5ex]
\mbox{(ii)} & \|Ax\|  \le M \left \| \varrho (G) \,  x \right \| &
\quad \mbox{for all $x \in X$} .
\end{eqnarray*}
\end{assumption1}
\begin{assumption2}
For some positive constants $E$ and $p$ we assume the solution
smoothness $x^\dagger - x_0 = G^p v$ with $v \in X$ and $\| v\| \le
E$. That is,
\[
x^\dagger \in M_{p,E} = \left \{ x \in X \, \bigl | \, \|x - x_0
\|_p \le E \right\} .
\]
\end{assumption2}
Assumption A1 characterizes the smoothing properties of the operator
$A$ relative to the operator $G$ and allows the study of problems
with finitely and infinitely smoothing operators $A$ in a unique
manner. Typical index functions in applications are power type index
functions $\varrho (t) = t^a$ for problems (\ref{f11}) with finitely
smoothing operators $A$ and index functions $\varrho (t) = \exp
(t^{-a} )$ where the inverse $\varrho^{-1}$ is of logarithmic type
for problems with infinitely smoothing operators $A$. Such problems
appear, e.\,g., in inverse heat conduction. Assumption A2
characterizes the smoothness of the element $x^\dagger - x_0$ in the
Hilbert scale $(X_r)_{r \in \mathbb R}$. By using Assumption A2 we
can study different smoothness situations for $x^\dagger - x_0$.

Let us give some comment on order optimal convergence rates for
identifying $x^\dagger$ from noisy data $y^\delta \in Y$ under the
link assumption A1 and the smoothness assumption A2. Let ${\mathcal
R}:Y \to X$ be an arbitrary method and ${\mathcal R} y^\delta$ be an
approximate solution for $x^\dagger$. Then, the quantity
\[
\Delta (\delta, {\mathcal R}) = \sup \bigl \{ \|{\mathcal R}
y^\delta -x^\dagger \| \, \bigl | \, y^\delta \in Y , \,
\|y-y^\delta\|\le \delta , \, x^\dagger \in M_{p,E} \bigr \}
\]
is called {\it worst case error} of the method ${\mathcal R}$ on the
set $M_{p,E}$.
An optimal method ${\mathcal R}_{\rm opt}$ is characterized by
$\Delta (\delta, {\mathcal R}_{\rm opt}) = \inf_{\mathcal R} \Delta
(\delta, {\mathcal R})$, and this quantity is called {\it best
possible worst case error} on the set $M_{p,E}$.
Under Assumption A1(ii) the best possible worst case error can be
estimated from below by
\begin{equation}\label{f211}
\inf_{\mathcal R} \Delta (\delta, {\mathcal R}) \ge
 E \,  \Bigl [ \psi_p^{-1} \Bigl ( \frac{\delta}{M E } \Bigr ) \Bigr ]^p
 \quad \text{with} \quad \psi_p (t):= t^p \varrho
 (t)
\end{equation}
provided $\delta/(M E)$ is an element of the spectrum of the
operator $\psi_p (G)$, see \cite{ta98-2} and \cite[proof of Theorem
2.2]{npt05}.
This lower bound will serve us as benchmark for the best possible
accuracy for identifying $x^\dagger$ from noisy data $y^\delta \in
Y$ under the link assumption A1 and the smoothness assumption A2.
%
\subsection{Exploiting operator monotonicity}
Operator monotonicity has been applied before in different papers
for deriving order optimal error bounds in regularization under
general smoothness assumptions, see, e.\,g.,
\cite{bhty06,mt06,nt08}. In this section we are going to derive some
consequences of Assumption A1 by using operator monotonicity of
certain functions. Considering the error representations (\ref{f28})
-- (\ref{f210}) we see that two types of estimates are helpful for
deriving error bounds for the noise amplification error  and the
regularization error in the $\|Ax\|$ -- norm and the $\|x\|_p$ --
norm, namely estimates of the type
\[
\|f_1(T^*T) G^{-s} x \| \le \| G^{-p} x \| \quad \mbox{and} \quad \|
G^{s-p} x \| \le \| f_2(T^*T) x \|
\]
with certain functions $f_1$, $f_2$ and  some constant $p> 0$ from
Assumption A2. We will derive such estimates from Assumption A1 by
using the concept of operator monotone functions which is based on
the concept of semiordering. Note that for two nonnegative,
self-adjoint bounded linear operators $S_1$, $S_2 \in {\mathcal L}
(X)$ the {\it semiordering} $S_1 \le S_2$ is defined by $(S_1 x,x)
\le (S_2 x,x)$ for all $x \in X$, or equivalently, by $\| S_1^{1/2}
x\| \le \| S_2^{1/2} x\|$ for all $x \in X$.
\begin{definition} \label{d23}
An index function $f: (0, a] \to \mathbb R$ is called {\it operator
monotone} if and only if for any pair of self-adjoint linear
operators $S_1$, $S_2$ with spectra in $(0,a]$, the relation $S_1
\le S_2$ implies the relation $f(S_1) \le f(S_2)$.
\end{definition}

Properties and examples for operator monotone functions may be found
in \cite{bh97,mp03b,uch10}. Our further study is based on several
functions. The first function is
\begin{equation}\label{f212}
\psi_r ( \lambda)= \lambda^r \varrho ( \lambda ), \quad \psi_r: (0,
a] \to ( 0, a^r \varrho (a) ]
\end{equation}
with  $\varrho$ from Assumption A1, $a = \|G\|$ and arbitrary
constant $r$ for which $\lambda^r \varrho (\lambda)$ is
monotonically increasing. Two other functions $h$ and $w$ are
\begin{equation}\label{f213}
h(t) = \left [  \psi_s^{-1} (\sqrt{t})   \right ]^{p-s} , \quad
w(t)= 1/h(t)
\end{equation}
with constant $s$ from (\ref{f21}) and $p$ from Assumption A2.

\begin{remark}\label{r24}
\rm The function $h$ defined by (\ref{f213}) possesses the following
properties:
\begin{enumerate}
\item[(i)] Due to the identity $\sqrt{t} h(t) = \psi_p \left ( (
\psi_s^2)^{-1} (t) \right )$, the function $t \to \sqrt{t} h(t)$ is
an index function and hence monotonically increasing in the both
cases $s \ge p$ and $s \le p$.
\item[(ii)] In the case of high order regularization with $s \ge p$,
the function $h$ is non-increasing. Hence, for $s \ge p$, the
function $t \to h(t) / \sqrt{t}$ is always monotonically decreasing.
\end{enumerate}
\end{remark}
In our next proposition we derive some estimates by using Assumption
A1.
\begin{proposition} \label{p25}
Let $h$ and $w$ be defined by (\ref{f213}). Then,
\begin{align}
& \| w ( T^*T/M^2 ) G^{-s} x \| \le \| G^{-p} x \| & &
\text{\hspace{-0.3cm} under A1(ii) and $w^2$ operator monotone},
\label{f214}
\\[1ex]
& \| w ( T^*T/m^2 ) G^{-s} x \| \le \| G^{-p} x \| & &
\text{\hspace{-0.3cm} under
A1(i) and $h^2$ operator monotone}, \label{f215} \\[1ex]
& \| G^{s-p} x \| \le  \| w ( T^*T/m^2  ) x  \| & &
\text{\hspace{-0.3cm} under A1(i) and $w^2$ operator monotone},
\label{f216}
\\[1ex]
& \| G^{s-p} x \| \le  \| w ( T^*T/M^2  ) x  \| & &
\text{\hspace{-0.3cm} under A1(ii) and $h^2$ operator monotone}.
\label{f217}
\end{align}
\end{proposition}
\begin{proof}
First, let us prove (\ref{f214}). It follows from Assumption A1(ii)
that
\[
\| AG^s x \| = \| Tx\| = \| (T^*T)^{1/2} x \| \le M \| \varrho (G)
G^s x \| ,
\]
which may be written in the equivalent form $ T^*T /M^2 \le
\varrho^2 (G) G^{2s}$. By using the function $\psi_s$ defined by
(\ref{f212}), this estimate can be written as $ T^*T /M^2 \le
\psi_s^2 (G)$.
Since $w^2:=1/h^2$ is assumed to be operator monotone and since $w^2
\left ( \psi_s^2 (G) \right ) = G^{2s-2p}$ we obtain that $ w^2
\left (  T^*T /M^2 \right ) \le G^{2s-2p}$ which gives $\| w \left (
 T^*T /M^2 \right ) x \| \le \| G^{s-p} x\| $ and hence
(\ref{f214}).
Second, let us prove (\ref{f215}). The link condition A1(i) may be
written as $ \psi_s^2 (G ) \le  T^*T /m^2 $.
Since $h^2$ is assumed to be operator monotone and since $h^2 \left
( \psi_s^2 (G) \right ) = G^{2p-2s}$ we obtain that $G^{2p-2s} \le
h^2 \left (  T^*T /m^2 \right ) $. Since $t \to -1/t$ is operator
monotone there follows that $w^2 \left (  T^*T /m^2 \right ) \le
G^{2s-2p}$, which gives (\ref{f215}). The proof of the estimates
(\ref{f216}) and (\ref{f217}) is similar.
\end{proof}
\begin{example}\label{e26} \rm
({\it Finitely smoothing case}). Let us assume that the operators
$A^*A$ and $G$ are related by
\begin{equation}\label{f218}
A^*A = G^{2a}
\end{equation}
where $a$ is some positive constant. In this case both Assumptions
A1(i) and A1(ii) hold true as equality with $\varrho (\lambda) =
\lambda^a$, $m=1$ and $M=1$. We easily see that the function
$\varrho$ is an index function and that the function $\psi_s$
defined in (\ref{f212}) attains the form $\psi_s (\lambda) =
\lambda^{a+s}$. Since $\psi_s^{-1} (\sqrt{t}) = t^{1/(2a+2s)}$ we
obtain that the functions $h$ and $w$ defined in (\ref{f213})
possess the representations
\[
h(t) = t^{\frac{p-s}{2(a+s)}} \,, \quad w(t) =
t^{\frac{s-p}{2(a+s)}} \,.
\]
Power functions $t^\nu$ are operator monotone for $0 \le \nu \le 1$,
see \cite{mp03b}. Hence, under the natural side conditions $p \ge
0$, $a>0$ and $s > -a$ we obtain

\vspace{0.2cm}

\begin{tabular}{ll}
(i) & that $w^2$ is an operator monotone function for $s \ge p$,
\\[0.5ex]
(ii) & that $h^2$ is an operator monotone function for $s \le p \le
2s+a$.
\end{tabular}
\end{example}
%
\subsection{Interpolation in variable Hilbert scales}
By interpolation in variable Hilbert scales we can estimate the
intermediate norm $\|x\|$ if estimates for some weaker norm $\|
\varrho (G) x \|$ and some stronger norm $\| x\|_r$ are known.
Variable Hilbert scale inequalities have been introduced by Hegland,
see \cite{he92,he95}. Such inequalities which extend the classical
interpolation inequality became a powerful tool in the analysis of
regularization under general smoothness conditions, see, e.g.,
\cite{mh08,mp03,npt05,nst03,rt09,ta98-2}. Variable Hilbert scale
interpolation is sometimes also called interpolation with a function
parameter, see \cite{bs07,mm08}. In our paper we are aiming to
combine special variable Hilbert scale inequalities with tools from
operator monotonicity.
\begin{proposition} \label{p27}
Assume $r \ge 0$, $\|x\|_r\le c_1$ and $\| \varrho (G) x \| \le c_2
$ with some index function $\varrho$ and constants $c_1$, $c_2$. Let
$\xi_r (t):= \psi_r^2 (t^{1/(2r)})$ be convex where $\psi_r$ is
given by (\ref{f212}). Then,
\begin{equation}\label{f219}
\| x \| \le c_1  \left [ \psi_r^{-1} \left ( \frac{c_2 }{c_1} \right
) \right ]^r .
\end{equation}
\end{proposition}
\begin{proof}
Let $E_\lambda$ the spectral family of $G^{-2r}$. Since $\xi_r$ is
convex we may employ Jensen's inequality and obtain
\[
\xi_r \left ( \frac{\|x\|^2}{\|x\|_r^2} \right )  \le \frac{ \int
\xi_r ( \lambda^{-1} ) \lambda \, \mbox{d} \| E_\lambda x\|^2}
{\|x\|_r^2}  = \frac{ \| \varrho (G) x\|^2 } {\|x\|_r^2}  ,
\]
or equivalently,
$ \|x\|_r \cdot \psi_r \left (\| x \|^{1/r}/\|x\|_r^{1/r}  \right )
\le \| \varrho (G) x \| . $
Since $\varrho (t) = t^{-r} \psi_r (t)$ is increasing we obtain that
$t \to t \psi_r (1/t^{1/r})$ is decreasing. Hence,
\begin{equation*}
c_1 \cdot \psi_r \left (\frac{\| x \|^{1/r}}{c_1^{1/r}}  \right )
\le \| x\|_r \cdot \psi_r \left (\frac{\| x \|^{1/r}}{\|
x\|_r^{1/r}} \right ) \le \| \varrho (G) x \| \le c_2 .
\end{equation*}
Rearranging terms gives (\ref{f219}).
\end{proof}
In our next proposition we provide a further estimate which is based
on interpolation arguments.

\begin{proposition}\label{p28}
Let $x_n$ be the regularized solution (\ref{f21}) with exact data
$y$, let $\varrho$ be an arbitrary index function, let Assumption A2
hold and assume $0 \le s \le p$. Let in addition $\psi_s$ be defined
by (\ref{f212}) and
\begin{equation}\label{f220}
 f(t):= \psi_s^2 \left (
t^{1/(2p-2s)} \right )
\end{equation}
be convex. Then, for all regularization parameters $\sigma_n$,
\begin{equation}\label{f221}
\| x_n - x^\dagger \|_{2s-p} \le  E \, \left [ \psi_p^{-1} \left (
\frac{\| \varrho (G) (x_n-x^\dagger )\|}{E } \right ) \right
]^{2p-2s} \,.
\end{equation}
\end{proposition}
\begin{proof}
Let us introduce the abbreviation $z=x^\dagger - x_n$. From
(\ref{f26}) we have the identity $G^{-s} z = r_n (T^*T) G^{-s}
(x^\dagger - x_0)$, and due to $r_n (\lambda) \le 1$ we have the
estimate $\| r_n^{1/2} (T^*T) \| \le 1$. We use these properties and
obtain due to Cauchy Schwarz inequality and Assumption A2 that
\begin{eqnarray}\label{f222}
\| z\|_s^2 & = & \| r_n (T^*T) G^{-s} ( x^\dagger - x_0) \|^2 \nonumber \\
& \le & \| r_n^{1/2} (T^*T) G^{-s} ( x^\dagger - x_0) \|^2 \nonumber \\
& = & (G^{p-2s} z, G^{-p} ( x^\dagger - x_0 ) ) \nonumber \\
& \le & E \|z\|_{2s-p} .
\end{eqnarray}
In the case $s=p$, (\ref{f221}) follows from (\ref{f222}). In the
case $0 \le s < p$, our next aim consists in deriving a second
estimate that relates the intermediate norm $\|z\|_{2s-p}$ with the
weaker norm $\|\varrho (G) z\|$ and the stronger norm $\|z\|_s$. We
derive this estimate by interpolation in variable Hilbert scales.
Since $f$ is convex we may employ Jensen's inequality and have
\[
f \left ( \frac{\|z \|_{2s-p}^2} {\|z\|_s^2} \right ) = f \left (
\frac{ \int  \lambda^{s-p} \cdot \lambda^s \, \mbox{d} \| E_\lambda
z \|^2 } {\int \lambda^s \, \mbox{d} \| E_\lambda  z \|^2 } \right )
\le \frac{ \int f ( \lambda^{s-p} ) \cdot \lambda^s \, \mbox{d} \|
E_\lambda  z  \|^2 } {\int \lambda^s \, \mbox{d} \| E_\lambda z\|^2
}
\]
where $E_\lambda$ is the spectral family of $G^{-2}$. Since
$f(\lambda^{s-p} ) \lambda^s = \varrho^2 (\lambda^{-1/2} )$ we
obtain
\begin{equation}\label{f223}
f \left ( \frac{\|z \|_{2s-p}^2} {\|z \|_s^2} \right )
 \le
\frac{ \int \varrho^2 ( \lambda^{-1/2} )  \, \mbox{d} \| E_\lambda z
\|^2 } {\|z \|_s^2 } = \frac{ \| \varrho (G) z  \|^2 } {\|z \|_s^2 }
.
\end{equation}
Now, let us eliminate $\| z\|_{s}$ in  estimate (\ref{f223}). We
write estimate (\ref{f222}) in the equivalent form
\begin{equation}\label{f224}
\| z \|_{2s-p}^{1/2} / E^{1/2} \le  \|z\|_{2s-p} / \|z\|_s
\end{equation}
and introduce two auxiliary functions $g$ and $r$ by
\begin{equation}\label{f225}
g(t):= t^{-2} f(t^2) \quad \mbox{and} \quad r(t) := t f(t) =
\psi_p^2 \left ( t^{1/(2p-2s)} \right ).
\end{equation}
Since $f$ is convex and $f(0)=0$, $g$ is monotonically increasing.
Hence, by (\ref{f224}), the monotonicity of $g$ and (\ref{f223}),
\[
g \left ( \frac{\| z\|_{2s-p}^{1/2} }{E^{1/2} } \right ) \le g\left
( \frac{\| z \|_{2s-p} }{ \| z \|_s} \right ) = \frac{\| z
\|_s^2}{\| z \|_{2s-p}^2 } \,
 f \left ( \frac{\| z \|_{2s-p}^2 }{ \| z \|_s^2} \right ) \le
\frac{\| \varrho (G) z\|^2}{\| z\|_{2s-p}^2} .
\]
Multiplying by $\| z\|_{2s-p}^2 /E^2$ gives
\[
 r \left ( \frac{\| z \|_{2s-p} }{ E }\right ) \le \frac{\| \varrho
(G) z \|^2 }{ E^2 } .
\]
Since the inverse $r^{-1}$ has the form $r^{-1} (\lambda) = [
\psi_p^{-1} (\sqrt{\lambda}) ]^{2p-2s}$, we obtain (\ref{f221}).
\end{proof}
%
%
\section{A priori parameter choice}
\setcounter{equation}{0}
In this section we make use of Proposition \ref{p25} for estimating
the total error in different norms in case the regularization
parameter $\sigma_n$ from (\ref{f13}) is chosen {\it a priori} by
\begin{equation}\label{f31}
\sigma_n^{-1} = \frac{\delta^2}{E^2} \left [ \psi_p^{-1} \left (
\frac{\delta}{mE} \right ) \right ]^{2 (s-p)} .
\end{equation}
We note that in the finitely smoothing case of Example \ref{e26} the
{\it a priori} parameter choice (\ref{f31}) attains the form
$\sigma_n^{-1} = m^2 \left ( \frac{\delta}{m E} \right )^{2 (s+a) /
(a+p)}$.
\subsection{Error bounds in the $\| Ax\|$ -- norm}
We start by providing error bounds for arbitrary $\sigma_n > 0$.

\begin{proposition}\label{p31}
Let $x_n^\delta$ be defined by (\ref{f21}), $h$ and $w$ be defined
by (\ref{f213}) and assume the solution smoothness A2.
\begin{enumerate}
\item[(i)] High order regularization ($s \ge p$): If $w^2:=1/h^2$ is
operator monotone, then under the link condition A1(ii),
\begin{equation}\label{f32}
\| Ax_n^\delta - Ax^\dagger\| \le \delta + E \sqrt{\sigma_n^{-1}}
h(\sigma_n^{-1}/ M^2) .
\end{equation}
\item[(ii)] Low order regularization ($s \le p$): If $h^2$ is operator
monotone and if $h(t)/\sqrt{t}$ is decreasing, then under the link
condition A1(i),
\begin{equation}\label{f33}
\| Ax_n^\delta - Ax^\dagger\| \le \delta + E \sqrt{\sigma_n^{-1}}
h(\sigma_n^{-1} / m^2) .
\end{equation}
\end{enumerate}
\end{proposition}
\begin{proof}
For estimating $\| Ax_n^\delta - Ax_n\|$ we use the error
representation (\ref{f27}) and obtain due to $\lambda g_n (\lambda)
\le 1$, see (\ref{f23}), the estimate
\begin{equation}\label{f34}
\| Ax_n^\delta - Ax_n\| = \| T g_n (T^*T) T^* (y^\delta - y) \| \le
\delta \sup_\lambda | \lambda g_n (\lambda) | \le \delta .
\end{equation}
For estimating $\| Ax_n - Ax^\dagger\|$ in the high order case $s
\ge p$ we use the error representation (\ref{f28}), exploit
Assumption A2 and estimate (\ref{f214}) which requires operator
monotonicity of $w^2$ and the second link condition A1(ii) and
obtain
\begin{equation} \label{f35}
\| Ax_n - Ax^\dagger \|  = \| T  r_n (T^*T) G^{-s} ( x^\dagger -
x_0) \|
 \le  \left \| T r_n (T^*T)  h \left (
 T^*T /M^2 \right ) \right \| \cdot E
\end{equation}
For estimating the norm term in (\ref{f35}) we distinguish two cases
$\lambda \le \sigma_n^{-1}$ and $\lambda \ge \sigma_n^{-1}$. In the
first case $\lambda \le \sigma_n^{-1}$ we use $g_n (\lambda) \ge 0$,
or equivalently $r_n (\lambda) \le 1$, exploit the increasing
behavior of $\sqrt{t} h(t)$ that holds true due to Remark \ref{r24}
(i) and obtain
\[
\sqrt{\lambda} \, r_n (\lambda)  h(\lambda/M^2) \le \sqrt{\lambda}
h(\lambda/M^2) \le \sqrt{\sigma_n^{-1}} h(\sigma_n^{-1}/M^2) .
\]
In the second case $\lambda \ge \sigma_n^{-1}$ we use $\lambda r_n
(\lambda) \le \sigma_n^{-1}$, see (\ref{f24}), exploit the
decreasing behavior of $h(t)/\sqrt{t}$ that holds true due to Remark
\ref{r24} (ii) and obtain
\[
\sqrt{\lambda} \, r_n (\lambda)  h(\lambda/M^2) \le \sigma_n^{-1}
h(\lambda/M^2)/\sqrt{\lambda}  \le \sqrt{\sigma_n^{-1}}
h(\sigma_n^{-1}/M^2) .
\]
From the both cases we obtain that (\ref{f35}) attains the form
\[
\| Ax_n - Ax^\dagger\|\le E \sqrt{\sigma_n^{-1}} h(\sigma_n^{-1}
/M^2 ).
\]
From this estimate and (\ref{f34}) we obtain (\ref{f32}).
For the proof of part (ii) we proceed analogously by exploiting
(\ref{f215}) instead of (\ref{f214}).
\end{proof}

\begin{remark}
Let us discuss the monotonicity condition in part (ii) of
Proposition \ref{p31} for the finitely smoothing case of Example
\ref{e26}. For this example we have
\[
h(t)/\sqrt{t} = t^{\frac{p-2s-a}{2(a+s)}} .
\]
Hence, $h(t)/\sqrt{t}$ is decreasing for $p \le 2s+a$. This
coincides with Natterer's side condition in (\ref{f14}).
\end{remark}

\begin{corollary}\label{c33}
Let be satisfied the assumptions of Proposition \ref{p31} and let
$\sigma_n$ be chosen by (\ref{f31}). Then, in the both cases (i) and
(ii) of Proposition \ref{p31} we have
\begin{eqnarray}
&& \mbox{\rm (i)} \qquad \| Ax_n^\delta - Ax^\dagger\| \le \left ( 1
+ M/m
\right ) \delta , \label{f36}\\[0.5ex]
&& \mbox{\rm (ii)} \qquad \| Ax_n^\delta - Ax^\dagger\| \le 2 \delta
. \label{f37}
\end{eqnarray}
\end{corollary}
\begin{proof}
Let us prove estimate (\ref{f36}) for the high order case (i). The
{\it a priori} parameter choice (\ref{f31}) can be written in the
equivalent form
\begin{equation}\label{f38}
E \sqrt{\sigma_n^{-1}} = \delta w(\sigma_n^{-1}/m^2) .
\end{equation}
Hence, by (\ref{f32}) and (\ref{f38}),
\begin{equation}\label{f39}
\| Ax_n^\delta - Ax^\dagger\| \le \delta + \delta
h(\sigma_n^{-1}/M^2) w(\sigma_n^{-1}/m^2) .
\end{equation}
Since $\sigma_n^{-1}/M^2 \le \sigma_n^{-1}/m^2$ and since $\sqrt{t}
h(t)$ is increasing we have
\[
h(\sigma_n^{-1}/M^2) \le \frac{M}{m} h(\sigma_n^{-1}/m^2).
\]
From this estimate and (\ref{f39}) we obtain (\ref{f36}). The proof
for the estimate (\ref{f37}) for the low order case (ii) is
analogous.
\end{proof}

\subsection{Error bounds in $X_p$}

We start by providing error bounds with respect to the $\| \cdot
\|_p $ -- norm for arbitrary $\sigma_n
> 0$.

\begin{proposition}\label{p34}
Let $x_n^\delta$ be defined by (\ref{f21}),  $h$ and $w$ be defined
by (\ref{f213}) and assume the link condition A1 and the solution
smoothness A2.
\begin{enumerate}
\item[(i)] High order regularization ($s \ge p$): If $w^2:=1/h^2$ is
operator monotone,
\begin{equation}\label{f310}
\| x_n^\delta - x^\dagger\|_p \le \delta \sqrt{\sigma_n} w
(\sigma_n^{-1}/m^2) + E \cdot M/m  .
\end{equation}
\item[(ii)] Low order regularization ($s \le p$): If $h^2$ is operator
monotone and if $h(t)/\sqrt{t}$ is decreasing,
\begin{equation}\label{f311}
\| x_n^\delta - x^\dagger\|_p \le \delta \sqrt{\sigma_n}
w(\sigma_n^{-1}/M^2) + E \cdot M/m .
\end{equation}
\end{enumerate}
\end{proposition}
\begin{proof}
Let us consider the high order case (i). For estimating $\|
x_n^\delta - x_n\|_p$ we use the error representation (\ref{f29})
and obtain due to the estimate (\ref{f216}) of Proposition \ref{p25}
the estimate
\begin{equation}\label{f312}
\| x_n^\delta - x_n\|_p = \| G^{s-p} g_n (T^*T) T^* (y^\delta - y)
\| \le \delta \left \| w \left ( \frac{1}{m^2} T^*T \right ) g_n
(T^*T) T^* \right \| .
\end{equation}
For estimating the norm term in (\ref{f312}) we distinguish two
cases $\lambda \le \sigma_n^{-1}$ and $\lambda \ge \sigma_n^{-1}$.
In the first case $\lambda \le \sigma_n^{-1}$ we use $g_n (\lambda)
\le \sigma_n$, see (\ref{f22}), exploit the increasing behavior of
$\sqrt{t} w(t)$ that follows since due to Remark \ref{r24} (ii) the
function $h(t) / \sqrt{t}$ is decreasing and obtain
\[
w (\lambda/m^2) g_n (\lambda) \sqrt{\lambda} \le w (\lambda/m^2)
\sqrt{\lambda} \, \sigma_n \le w(\sigma_n^{-1}/m^2) \sqrt{\sigma_n}
.
\]
In the second case $\lambda \ge \sigma_n^{-1}$ we use $\lambda g_n
(\lambda) \le 1$, exploit the decreasing behavior of $w(t)/\sqrt{t}$
that follows since due to Remark \ref{r24} (i) the function
$\sqrt{t} h(t)$ is increasing and obtain
\[
w (\lambda/m^2) g_n (\lambda) \sqrt{\lambda} \le w
(\lambda/m^2)/\sqrt{\lambda} \le w(\sigma_n^{-1}/m^2)
\sqrt{\sigma_n} .
\]
From the both cases we obtain that (\ref{f312}) attains the form
\begin{equation}\label{f313}
\| x_n^\delta - x_n\|_p \le \delta \sqrt{\sigma_n}
w(\sigma_n^{-1}/m^2) .
\end{equation}
For estimating $\| x_n - x^\dagger\|_p$ in the high order case (i)
we use the error representation (\ref{f210}), exploit the estimate
(\ref{f214}) of Proposition \ref{p25}, use in addition the estimate
(\ref{f216}) and obtain due to $\|x^\dagger - x_0\|_p \le E$ and
$g_n (\lambda) \ge 0$, or equivalently $r_n (\lambda) \le 1$, the
estimate
\begin{eqnarray}\label{f314}
\| x^\dagger - x_n \|_p & = & \| G^{s-p} r_n (T^*T) G^{-s}
(x^\dagger - x_0) \| \nonumber \\
& \le & \left \| w \left ( T^*T / m^2 \right ) r_n (T^*T)
G^{-s} ( x^\dagger - x_0)  \right \| \nonumber \\
& \le & E \left \| w \left (  T^*T /m^2 \right ) r_n (T^*T) h \left
( T^*T /M^2 \right ) \right \|
\nonumber \\
& \le & E \sup_\lambda | h(\lambda/M^2) w (\lambda/m^2) | .
\end{eqnarray}
Due to Remark \ref{r24} (i) the function $\sqrt{t} h(t)$ is
increasing, or equivalently, $w(t)/\sqrt{t}$ is decreasing. Hence,
from $\lambda/M^2 \le \lambda/m^2$  we have $w(\lambda/m^2) \le
\frac{M}{m} w(\lambda/M^2)$, and (\ref{f314}) attains the form
\begin{equation}\label{f315}
\| x^\dagger - x_n \|_p \le   E \cdot M/m  .
\end{equation}
From (\ref{f313}) and (\ref{f315}) we obtain (\ref{f310}).
For the proof of part (ii) we proceed analogously by exploiting
(\ref{f215}) and (\ref{f217}) instead of (\ref{f214}) and
(\ref{f216}).
\end{proof}

From Proposition \ref{p34} we have along the line of Corollary
\ref{c33} the following

\begin{corollary}\label{c35}
Let be satisfied the assumptions of Proposition \ref{p34} and let
$\sigma_n$ be chosen by (\ref{f31}). Then, in the both cases (i) and
(ii) of Proposition \ref{p34} we have
\begin{eqnarray}
&& \mbox{\rm (i)} \qquad \| x_n^\delta - x^\dagger\|_p \le E \cdot
\left ( 1 + M/m \right ) , \label{f316}\\[0.5ex]
&& \mbox{\rm (ii)} \qquad \| x_n^\delta - x^\dagger\|_p \le  2E
\cdot M/m . \label{f317}
\end{eqnarray}
\end{corollary}


\subsection{Error bounds in $X$}

For deriving order optimal error bounds for the total error
$\|x_n^\delta - x^\dagger\|$ with $\sigma_n$ chosen a priori by
(\ref{f31}) we employ interpolation techniques from Proposition
\ref{p27} and use the results of Corollary \ref{c35} which provides
a bound for $\| x_n^\delta - x^\dagger \|_p$ and the results of
Corollary \ref{c33} which together with the first link condition
A1(i) provides a bound for $\| \rho (G) (x_n^\delta - x^\dagger)
\|$.

\begin{theorem}\label{t36}
Let be satisfied the assumptions of Proposition \ref{p31} and
\ref{p34} and let $\sigma_n$ be chosen a priori by (\ref{f31}). If
the function  $\xi_p(t):= \psi_p^2 (t^{1/(2p)})$ is convex, then
$x_n^\delta$ is order optimal on the set $M_{p,E}$ in the both cases
of high order regularization $ s \ge p$ and low order regularization
$s \le p$. In fact, in both cases,
\begin{equation}\label{f318}
\| x_n^\delta - x^\dagger\| \le c_1   \left [ \psi_p^{-1} \left (
c_2 \delta  \right ) \right ]^p
\end{equation}
with some constants $c_1$ and $c_2$ which can be extracted from the
proof.
\end{theorem}
\begin{proof}
Due to Corollary \ref{c35}, Corollary \ref{c33} and Assumption
A1(i), in both cases of high- and low order regularization we have
\[
\| x_n^\delta - x^\dagger\|_p \le k_1 \quad \mbox{and} \quad \|
\varrho (G) (x_n^\delta - x^\dagger) \| \le k_2 \delta
\]
with some constants $k_1$ and $k_2$. Using the interpolation
estimate (\ref{f219}) of Proposition \ref{p27} yields (\ref{f318}).
\end{proof}

Note that for the finitely smoothing case of Example \ref{f26} we
have $\xi_p (t) = t^{(a+p)/p}$, which is convex for arbitrary $p>0$.

\subsection{Revisiting the low order case}

The error bounds given in Subsection 3.3 require in both cases of
high order and low order regularization the both link conditions
A1(i) and A1(ii). We will show in this subsection that in the case
of low order regularization $s \le p$  order optimal error bounds
can be obtained without the second link condition A1(ii). However,
this will only be possible for $s \ge 0$.
We exploit in our study the property
\begin{equation}\label{f319}
\sqrt{\lambda} g_n (\lambda) \le \sqrt{\sigma_n} \quad \mbox{for}
\quad \lambda \in \left ( 0, \|T\|^2 \right ]
\end{equation}
which follows from the both properties (\ref{f22}) and (\ref{f23})
of Proposition \ref{p21} and start by providing some error bound in
the $\| \cdot \|_s$ -- norm for arbitrary $\sigma_n > 0$.

\begin{proposition}\label{p37}
Let $x_n^\delta$ be defined by (\ref{f21}), $h$ be defined by
(\ref{f213}) and assume the solution smoothness A2.
If $h^2$ is operator monotone and $h(t)/t$ is decreasing, then under
the link condition A1(i),
\begin{equation}\label{f320}
\| x_n^\delta - x^\dagger\|_s \le  \delta \, \sqrt{\sigma_n} + E
h(\sigma_n^{-1} / m^2)  .
\end{equation}
\end{proposition}
\begin{proof}
For estimating $\| x_n^\delta - x_n\|_s$ in the low order case $s
\le p$ we use the error representation (\ref{f26}) and obtain due to
$\sqrt{\lambda} g_\alpha (\lambda) \le \sqrt{\sigma_n}$, see
(\ref{f319}), the estimate
\begin{equation}\label{f321}
\| x_n^\delta - x_n\|_s = \| g_n (T^*T) T^* (y^\delta - y) \| \le
\delta \,  \sqrt{\sigma_n} .
\end{equation}
For estimating $\| x_n - x^\dagger\|_s$  we use the error
representation (\ref{f26}), exploit the estimate (\ref{f215}) of
Proposition \ref{p25} and obtain due to $\|x^\dagger - x_0\|_p \le
E$ the estimate
\begin{equation}\label{f322}
\| x^\dagger - x_n \|_s = \| r_n (T^*T) G^{-s} (x^\dagger - x_0) \|
\le E \left \|r_n (T^*T)  h \left ( T^*T / m^2 \right ) \right \| .
\end{equation}
For estimating the norm term in (\ref{f322}) we distinguish two
cases $\lambda \le \sigma_n^{-1}$ and $\lambda \ge \sigma_n^{-1}$.
In the first case $\lambda \le \sigma_n^{-1}$ we use $r_n (\lambda)
\le 1$, or equivalently, $g_n (\lambda) \ge 0$ exploit the
increasing behavior of $h(t)$ which is always satisfied since $h^2$
is operator monotone and obtain
\[
r_n (\lambda) h(\lambda/m^2) \le h(\lambda/m^2) \le
h(\sigma_n^{-1}/m^2 ) .
\]
In the second case $\lambda \ge \sigma_n^{-1}$ we use $\lambda r_n
(\lambda)  \le \sigma_n^{-1}$, exploit the decreasing behavior of
$h(t)/t$ and obtain
\[
r_n (\lambda)  h(\lambda/m^2) \le \sigma_n^{-1} h(\lambda/m^2 ) /
\lambda \le h(\sigma_n^{-1}/m^2)  .
\]
From the both cases we obtain that (\ref{f322}) attains the form
$\| x^\dagger - x_n \|_s \le E h(\sigma_n^{-1}/m^2)$. From this
estimate and (\ref{f321}) we obtain (\ref{f320}).
\end{proof}

Since the parameter choice (\ref{f31}) can be written in the
equivalent form $\delta \sqrt{\sigma_n} = E h(\sigma_n^{-1}/m^2)$ we
obtain from Proposition \ref{p37}  the following

\begin{corollary}\label{c38}
Let be satisfied the assumptions of Proposition \ref{p37} and let
$\sigma_n$ be chosen a priori by (\ref{f31}). Then,
\begin{equation}\label{f323}
 \| x_n^\delta - x^\dagger\|_s \le  2 E \left [ \psi_p^{-1} \left (
 \frac{\delta}{m E} \right ) \right ]^{p-s} .
\end{equation}
\end{corollary}
\begin{proof}
From (\ref{f212}) we have $\varrho (\lambda) = \lambda^{-s} \psi_s
(\lambda)$ and $\varrho (\lambda) = \lambda^{-p} \psi_p (\lambda)$.
Consequently,
\begin{equation}\label{f324}
\psi_s (\lambda) = \lambda^{s-p} \psi_p (\lambda) \quad \mbox{and}
\quad \psi_p (\lambda) = \lambda^{p-s} \psi_s (\lambda) .
\end{equation}
We use the first equation of (\ref{f324}), substitute $\lambda=
\psi_p^{-1} (t)$ and obtain $\psi_s ( \psi_p^{-1} (t) ) = t [
\psi_p^{-1} (t) ]^{s-p}$. From this equation we conclude that the
{\it a priori} parameter choice (\ref{f31}), which is equivalent to
$ \sigma_n^{-1/2}/m = \frac{\delta}{m E} \left [ \psi_p^{-1} \left (
\frac{\delta}{m E} \right ) \right ]^{s-p}$,
can be rewritten as $\sigma_n^{-1/2} /m = \psi_s \left ( \psi_p^{-1}
\left ( \frac{\delta}{mE}\right ) \right ) $, or equivalently,
\begin{equation}\label{f325}
\psi_s^{-1} \left ( \frac{\sigma_n^{-1/2}}{m} \right ) =\psi_p^{-1}
\left ( \frac{\delta}{m E} \right ) .
\end{equation}
Clearly, (\ref{f325}) is equivalent to $\psi_p \left ( \psi_s^{-1}
\left ( \sigma_n^{-1/2}/m \right ) \right ) = \frac{\delta}{mE}$. We
use the second equation of (\ref{f324}) and write this equation in
the form
\[
\left [ \psi_s^{-1} \left ( \frac{\sigma_n^{-1/2}}{m} \right )
\right ]^{p-s}  \frac{\sigma_n^{-1/2}}{m} = \frac{\delta}{m E}.
\]
From this equation and the definition of $h$ by (\ref{f213}) we see
that the parameter choice (\ref{f31}) can be written in the
equivalent form $\delta \sqrt{\sigma_n} = E h(\sigma_n^{-1}/m^2)$.
Due to this equation, estimate (\ref{f320}) of Proposition \ref{p37}
attains the form
\begin{equation}\label{f326}
\| x_n^\delta - x^\dagger\|_s \le   2 \delta \, \sqrt{\sigma_n}  .
\end{equation}
From this estimate and the parameter choice (\ref{f31}) we obtain
(\ref{f323}).
\end{proof}

Now, by using the both estimates (\ref{f323}) and (\ref{f37}), we
obtain the following order optimality result for $x_n^\delta$ on the
set $M_{p,E}$.

\begin{theorem}\label{t39}
Let $x_n^\delta$ be defined by (\ref{f21}) with $\sigma_n$ chosen by
(\ref{f31}), assume the link condition A1(i) and the solution
smoothness A2.
If  $h^2$ is operator monotone, $h(t)/\sqrt{t}$ is decreasing and
$\xi_s(t):= \psi_s^2 (t^{1/(2s)})$ is convex, then
\begin{equation}\label{f327}
\| x_n^\delta - x^\dagger\| \le 2 E  \left [ \psi_p^{-1} \left (
\frac{\delta}{m E} \right ) \right ]^p .
\end{equation}
\end{theorem}
\begin{proof}
From estimate (\ref{f37}) and Assumption A1(i) we have the estimate
\begin{equation}\label{f328}
\| \varrho (G) (x_n^\delta - x^\dagger ) \| \le 2\delta/m .
\end{equation}
We apply the interpolation estimate (\ref{f219}) of Proposition
\ref{p27} and obtain together with (\ref{f326}) and (\ref{f328}) the
estimate
\begin{equation}\label{f329}
\| x_n^\delta - x^\dagger \| \le 2 \delta \sigma_n^{1/2} \left [
\psi_s^{-1} \left ( \sigma_n^{-1/2}/m \right ) \right ]^s .
\end{equation}
It remains to show that for the parameter choice (\ref{f31}) both
right hand sides of (\ref{f327}) and (\ref{f329}) coincide. We use
formula (3.25) and obtain that (\ref{f329}) can be written in the
equivalent form
\begin{equation}\label{f330}
\| x_n^\delta - x^\dagger \| \le 2 \delta \sigma_n^{1/2} \left [
\psi_p^{-1} \left ( \frac{\delta}{mE} \right ) \right ]^s .
\end{equation}
Now we rewrite (3.1) as $\sigma_n^{1/2} = \frac{E}{\delta} \left [
\psi_p^{-1} \left ( \frac{\delta}{mE} \right ) \right ]^{p-s}$,
substitute this into (\ref{f330}) and obtain (\ref{f327}).
\end{proof}

\section{Discrepancy principle}
\setcounter{equation}{0}
If the constants $m$ and $E$ in the {\it a priori} parameter choice
(\ref{f31}) are unknown, then the parameter choice
$ \sigma_n^{-1} = \left ( \frac{\delta}{c_2} \right )^2 \left [
\psi_p^{-1} \left ( \frac{\delta}{c_1 c_2} \right ) \right
]^{2(s-p)} $
may be used where $c_1$ and $c_2$ are positive constants guessing
$m$ and $E$, respectively. For this parameter choice, the order
optimality results of Theorems \ref{t36} and \ref{t39} still hold
true. In case of rough estimates for $m$ and $E$, and in particular
in cases where $p$ and $\psi_p$ are unknown, {\it a posteriori}
rules for choosing $ \sigma_n$ have to be used. In the discrepancy
principle (see \cite{mo66}) the regularization parameter $\sigma_n$
is chosen as the solution of the nonlinear equation
\begin{equation}\label{f41}
d (\sigma_n):= \| Ax_n^\delta - y^\delta \| = C \delta
\end{equation}
with some constant $C \ge  1$. For practical reasons it makes sense
to choose $\sigma_n$ such that
\begin{equation}\label{f42}
C_1 \delta \le d (\sigma_n) \le C_2 \delta
\end{equation}
with some constants $C_1$, $C_2$ that obey  $1 \le C_1 \le C_2$. In
computations it makes sense to choose $C_2$ with $C_2 > C_1$.

\begin{remark}\label{r41}
For realizing the discrepancy principle (\ref{f41}) or (\ref{f42})
approximately, one practical way is as follows. We start with some
large $\alpha_1$ in (\ref{f12}), use a decreasing $\alpha$-sequence
and iterate as long as the discrepancy is in the magnitude of the
noise level. More accurately, we consider the decreasing sequence
$\Delta = \{ \alpha_k \}_{k=1}^\infty$ and choose $n$ as the first
integer for which
\begin{equation}\label{f43}
\| A x_n^\delta - y^\delta \| \le C \delta < \| A x_k^\delta -
y^\delta \| , \quad 0 \le k < n
\end{equation}
with some $C > 1$. Some care is required for the final iteration
step in which one has to take care that the discrepancy becomes not
too small and remains in the magnitude of $\delta$. This can be
guaranteed by assuming that the final $\alpha_n$ is not too small
and obeys
\begin{equation}\label{f44}
1/\alpha_n \le c \sigma_{n-1}
\end{equation}
with some positive constant $c$. For the geometric sequence $\Delta
=  \{ q^{k-1} \alpha_1  \}_{k=1}^\infty$ with some $q < 1$,
assumption (\ref{f44}) is satisfied with $c=1/q$, see \cite{hg98}.
We show in Subsection 4.3 that for the version (\ref{f43}) of the
discrepancy principle, analogous convergence rate results to that of
the {\it a posteriori} rule (\ref{f42}) hold true.
\end{remark}
\subsection{Properties}
Due to (\ref{f21}), the discrepancy $y^\delta - A x_n^\delta$ can be
represented by
\begin{equation}\label{f45}
y^\delta - A x_n^\delta = r_n (TT^*) (y^\delta - A x_0 ) = \left (
\prod_{k=1}^n \alpha_k (TT^* + \alpha_k I)^{-1} \right ) (y^\delta -
A x_0 ) .
\end{equation}
From this representation we conclude that the discrepancy is
monotonically decreasing with respect to the iteration number, that
is,
\[
\|y^\delta -  A x_k^\delta \| < \|y^\delta -  A x_{k-1}^\delta \|,
\quad k=1,2,...
\]
For $\sigma_n \to \infty$ we have $r_n (\lambda) \to 0$, and for
$\sigma_n \to 0$ we have $r_n (\lambda) \to 1$. Therefore, by
(\ref{f45}), we have the two limit relations
\[
\lim_{\sigma_n \to \infty} \|y^\delta -  A x_n^\delta \| = 0 \quad
\mbox{and} \quad \lim_{\sigma_n \to 0} \|y^\delta -  A x_n^\delta \|
= \|y^\delta -  A x_0\|.
\]
From both limit relations we conclude that under the condition $\|
y^\delta - A x_0\| > C \delta$ there exists $\sigma_n$ (not
necessarily unique) that obeys  rule (\ref{f41}) or rule
(\ref{f43}), respectively,
and that under the condition $\| y^\delta - A x_0\| > C_2 \delta$
there exists $\sigma_n$  that obeys rule (\ref{f42}).

Now we assume that for some given $\sigma_{n-1}$ we have $\| A
x_{n-1}^\delta - y^\delta \| > C \delta$. Then, the discrepancy
$d(\alpha_n) := \| A x_{n}^\delta - y^\delta \|$ as a function of
$\alpha_n$ possesses following properties:
\begin{enumerate}
\item[(i)] For $\alpha_n \to 0$ we have the limit relation $\lim_{\alpha_n \to 0} d(\alpha_n) = 0$.
\item[(ii)] For $\alpha_n \to \infty$ we have  $\lim_{\alpha_n \to \infty} d(\alpha_n) = \| A
x_{n-1}^\delta - y^\delta \| > C \delta$.
\item[(iii)] The function $d (\alpha_n)$ is continuous and strictly
monotonically increasing.
\end{enumerate}
As a consequence, there exists $\alpha_n^*$ with $d(\alpha_n^*) = C
\delta$.

Following proposition gives us some monotonicity property for the
error $\| x_n^\delta - x^\dagger\|_s$ with respect to the
$X_s$--norm which tells us that the iteration (\ref{f12}) should not
be stopped as long as $\| A x_n^\delta - y^\delta \| \ge \delta$
holds. In the special case $s=0$, such monotonicity property may be
found in \cite{ht01}.

\begin{proposition}\label{p42}
Let $x^\dagger \in X_s$, let $x_n^\delta$ be defined by the
iteration (\ref{f12}) and let $\| A x_n^\delta - y^\delta \| \ge
\delta$. Then,
\begin{equation}\label{f46}
\| x_n^\delta - x^\dagger\|_s < \| x_{n-1}^\delta - x^\dagger\|_s .
\end{equation}
\end{proposition}
\begin{proof}
The iteration (\ref{f12}) can be rewritten as
\[
x_n^\delta = x_{n-1}^\delta + B^{-2s} A^* z_{n-1} \quad \mbox{with}
\quad z_{n-1} = (TT^* + \alpha_n I)^{-1} (y^\delta - A
x_{n-1}^\delta ) .
\]
Consequently, for $d:=\| x_n^\delta - x^\dagger\|_s^2 - \|
x_{n-1}^\delta - x^\dagger\|_s^2$ we have
\begin{align*}
d & =  \| x_{n-1}^\delta + B^{-2s} A^* z_{n-1} - x^\dagger\|_s^2
- \| x_{n-1}^\delta - x^\dagger\|_s^2 \nonumber \\
& =  \left ( 2 x_{n-1}^\delta - 2 x^\dagger + B^{-2s} A^* z_{n-1} ,
B^{-2s} A^* z_{n-1} \right )_s \nonumber \\
& = \left ( x_{n-1}^\delta + x_n^\delta - 2 x^\dagger ,
B^{-2s} A^* z_{n-1} \right )_s \nonumber \\
& = \left (A x_{n-1}^\delta + A x_n^\delta - 2 y ,
z_{n-1} \right ) \nonumber \\
& = \left (2 (y^\delta - y) + (A x_{n-1}^\delta - y^\delta) + (A
x_n^\delta - y^\delta)  ,
z_{n-1} \right ) \nonumber \\
& \le 2 \| z_{n-1} \|  \left ( \delta - \frac{\left ( (y^\delta - A
x_{n-1}^\delta ) + (y^\delta - A x_n^\delta), z_{n-1} \right ) }{2
\| z_{n-1} \|} \right ).
\end{align*}
Let $r_n:= y^\delta - A x_n^\delta$. Then, from (\ref{f45}) we have
$r_n = \alpha_n (TT^* + \alpha_n I)^{-1} r_{n-1}$. Hence, the
element $z_{n-1}$ can be written as $z_{n-1} = \alpha_n^{-1} r_n$.
Consequently,
\begin{equation}\label{f47}
\| x_n^\delta - x^\dagger\|_s^2 - \| x_{n-1}^\delta -
x^\dagger\|_s^2 \le \frac{2 \| r_n \|}{ \alpha_n} \left ( \delta -
\frac{\left ( r_{n-1} + r_n , r_n \right ) }{2 \| r_n \|} \right ) .
\end{equation}
We use again the identity $r_n = \alpha_n (TT^* + \alpha_n I)^{-1}
r_{n-1}$, or equivalently, $r_{n-1} = \alpha_n^{-1} (TT^* + \alpha_n
I) r_{n}$, multiply by $r_n$ and obtain
\begin{equation} \label{f48}
(r_{n-1} , r_n ) = \alpha_n^{-1} \| T^* r_n \|^2 + \| r_n \|^2 > \|
r_n\|^2 .
\end{equation}
From (\ref{f47}), (\ref{f48}) and $\| r_n\| \ge \delta$ we obtain
$\| x_n^\delta - x^\dagger\|_s^2 - \| x_{n-1}^\delta -
x^\dagger\|_s^2 < 0$.
\end{proof}

\subsection{Error bounds in $X$}
In this subsection we show that for $\sigma_n$  chosen by
(\ref{f41}) or (\ref{f42}), respectively, the order optimal error
bound (\ref{f318}) holds true under analogous assumptions of Theorem
\ref{t36}. In a first proposition we provide some estimate for the
regularization parameter $\sigma_n$ chosen by (\ref{f42}).

\begin{proposition}\label{p43}
Let $x_n^\delta$ be defined by (\ref{f21}), $h$ and $w$ be defined
by (\ref{f213}), $\sigma_n$ be chosen by the discrepancy principle
(\ref{f42}) with $1 < C_1 \le C_2$ and assume the solution
smoothness A2.
\begin{enumerate}
\item[(i)] High order regularization ($s \ge p$): If $w^2:=1/h^2$ is
operator monotone, then under the link condition A1(ii),
\begin{equation}\label{f49}
(C_1-1) \delta \le E \sqrt{\sigma_n^{-1}} h (\sigma_n^{-1}/M^2) .
\end{equation}
\item[(ii)] Low order regularization ($s \le p$): If $h^2$ is
operator monotone and if $h(t) /\sqrt{t}$ is decreasing, then under
the link condition A1(i),
\begin{equation}\label{f410}
(C_1-1) \delta \le E \sqrt{\sigma_n^{-1}} h (\sigma_n^{-1}/m^2) .
\end{equation}
\end{enumerate}
\end{proposition}
\begin{proof}
Let us prove part (i). From (\ref{f21}) we have $y^\delta -
Ax_n^\delta = r_n (TT^*) ( y^\delta - A x_0) $. Due to rule
(\ref{f42}), the identity $y - A x_n = r_n (TT^*) (y-Ax_0)$ and the
estimate $\| r_n (TT^*) \| \le 1 $ we obtain that
\begin{eqnarray} \label{f411}
C_1 \delta & \le & \| r_n (TT^*)  ( y^\delta - A x_0) \| \nonumber \\
& \le & \| r_n (TT^*) ( y - A x_0) \| + \|r_n (TT^*)  (y-y^\delta)
\| \nonumber \\
&  \le &  \|y - A x_n \| + \delta .
\end{eqnarray}
From the proof of Proposition \ref{p31} we have that
\[
\| y - A x_n \| \le E \sqrt{\sigma_n^{-1}} h (\sigma_n^{-1}/M^2 ) .
\]
This estimate and (\ref{f411}) provide the desired estimate
(\ref{f49}) and the proof of part (i) is complete. For the proof of
part (ii) we proceed in an analogous way, but use instead of
(\ref{f214}) the estimate (\ref{f215}) which requires the link
condition A1(i) and the operator monotonicity of the function $h^2$.
\end{proof}

From Propositions \ref{p34} and \ref{p43} we obtain that the total
error $x_n^\delta - x^\dagger$ is bounded in the $\| \cdot \|_p$ --
norm for the {\it a posteriori} parameter choice $\sigma_n$ chosen
by the discrepancy principle (\ref{f42}).

\begin{proposition}\label{p44}
Let $x_n^\delta$ be defined by (\ref{f21}), $\sigma_n$ be chosen by
the discrepancy principle (\ref{f42}) with $1 < C_1 \le C_2$ and
assume the solution smoothness A2 and the link condition A1.
\begin{enumerate}
\item[(i)] High order regularization ($s \ge p$):
If $w^2:=1/h^2$ is operator monotone,
\begin{equation}\label{f413}
\| x_n^\delta - x^\dagger \|_p \le \frac{E}{C_1 -1} \cdot
\frac{M}{m} + E \cdot \frac{M}{m} .
\end{equation}
\item[(ii)] Low order regularization ($s \le p$):
If $h^2$ is operator monotone and $h(t)/\sqrt{t}$ is decreasing,
then,
\begin{equation}\label{f414}
\| x_n^\delta - x^\dagger \|_p \le \frac{E}{C_1 -1} \cdot
\frac{M}{m} + E \cdot \frac{M}{m} .
\end{equation}
\end{enumerate}
\end{proposition}
\begin{proof}
In the case (i) we exploit the increasing behavior of $\sqrt{t}
h(t)$ and conclude from $\sigma_n^{-1}/M^2 \le \sigma_n^{-1}/m^2$
that $h(\sigma_n^{-1}/M^2) \le \frac{M}{m} h(\sigma_n^{-1}/m^2)$,
which together with part (i) of the two Propositions \ref{p34} and
\ref{p43} provides (\ref{f413}).
In the case (ii) we exploit the decreasing behavior of
$h(t)/\sqrt{t}$, or equivalently the increasing behavior of
$\sqrt{t} w(t)$ and conclude from $\sigma_n^{-1}/M^2 \le
\sigma_n^{-1}/m^2$ that $w(\sigma_n^{-1}/M^2) \le \frac{M}{m}
w(\sigma_n^{-1}/m^2)$, which together with part (ii) of the two
Propositions \ref{p34} and \ref{p43} provides (\ref{f414}).
\end{proof}

Now we are in a position to prove the main result of this section.
In our next theorem we will see that order optimal error bounds can
be guaranteed in case $\sigma_n$ is chosen by the discrepancy
principle (\ref{f42}) with $1 < C_1 \le C_2$.

\begin{theorem}\label{t45}
Let be satisfied the assumptions of Proposition \ref{p44} and assume
in addition that $\xi_p(t):= \psi_p^2 (t^{1/(2p)})$ is convex. Then,
\begin{equation}\label{f415}
\| x_n^\delta - x^\dagger\| \le c_1   \left [ \psi_p^{-1} \left (
c_2 \delta  \right ) \right ]^p
\end{equation}
with some constants $c_1$, $c_2$ which can be extracted from the
proof.
\end{theorem}
\begin{proof}
Due to Proposition \ref{p44}, in  both cases (i) and (ii) of high-
and low order regularization the total error obeys
\begin{equation}\label{f416}
\|x_n^\delta - x^\dagger \|_p \le c E
\end{equation}
with some $c \ge 1$ and  $\sigma_n$ chosen by the discrepancy
principle (\ref{f42}) with $1 < C_1 \le C_2$. From (\ref{f42}) and
the triangle inequality we have
\[
\| A x_n^\delta - A x^\dagger \| \le \| A x_n^\delta - y^\delta \| +
\| y-y^\delta\| \le (C_2 +1) \delta .
\]
Using in addition the link condition A1(i) yields
\begin{equation}\label{f417}
\|\rho (G)( x_n^\delta - x^\dagger ) \|  \le (C_2 +1) \delta / m .
\end{equation}
Now the result of the theorem follows from (\ref{f416}),
(\ref{f417}) and Proposition \ref{p27}.
\end{proof}


\subsection{Error bounds for rule (\ref{f43})}
For the {\it a posteriori} rule (\ref{f43}) of choosing the
regularization parametr $\sigma_n$, analogous order optimal error
bounds to that of Theorem \ref{t45} can be obtained.

\begin{theorem}\label{t46}
Let $x_n^\delta$ be defined by (\ref{f21}), let $\sigma_n$ be chosen
by rule (\ref{f43}) where $\alpha_n$ obeys (\ref{f44}), let the both
Assumptions A1 and A2 hold and assume that $\xi_p (t):= \psi_p^2
\left ( t^{1/(2p)}\right )$ is convex. Assume further
\begin{enumerate}
\item[(i)] in case of high order regularization ($s \ge p$) that $w^2:=
1/h^2$ is operator monotone and
\item[(ii)] in case of low order regularization ($s \le p$) that $h^2$ is operator monotone and
$h(t) / \sqrt{t}$ is decreasing.
\end{enumerate}
Then, the regularized solution $x_n^\delta$ obeys the order optimal
error bound
\begin{equation}\label{f418}
\| x_n^\delta - x^\dagger \| \le c_1 \left [ \psi_p^{-1} ( c_2
\delta ) \right ]^p
\end{equation}
with some constants $c_1$, $c_2$ which can be extracted from the
proof.
\end{theorem}
\begin{proof}
We give the proof for the high order case $s \ge p$, the proof for
the low order case $s \le p$ is similar. In the {\it first step} of
our proof we proceed according to the proof of Proposition
\ref{p43}, exploit that $C \delta \le \| r_{n-1} (TT^*) (y^\delta -
A x_0)\|$ and obtain
\[
(C-1) \delta \le E \sqrt{\sigma_{n-1}^{-1}} \, h \left (
\sigma_{n-1}^{-1} /M^2 \right ) .
\]
From this estimate and (\ref{f310}) we obtain
\begin{equation}\label{f419}
\| x_n^\delta - x^\dagger\|_p \le \frac{E}{C-1} \cdot
\frac{\sqrt{\sigma_{n-1}^{-1}} \, h \left ( \sigma_{n-1}^{-1} /M^2
\right ) }{\sqrt{\sigma_{n}^{-1}} \, h \left ( \sigma_{n}^{-1} /m^2
\right ) } + E \cdot M/m .
\end{equation}
Now we consider two cases. In the {\it first case} with
$\sigma_{n-1}^{-1} /M^2 \le \sigma_n^{-1}/m^2$ we use the increasing
behavior of $\sqrt{t} h(t)$ and obtain from (\ref{f419}) the
estimate
\[
\| x_n^\delta - x^\dagger \|_p \le \frac{E}{C-1} \cdot \frac{M}{m} +
E \cdot M/m .
\]
In the {\it second case} with $\sigma_{n-1}^{-1} /M^2 \ge
\sigma_n^{-1}/m^2$ we use the decreasing behavior of $h$, exploit in
addition that due to (\ref{f44}) we have $\sigma_n = 1/\alpha_n +
\sigma_{n-1} \le (c+1) \sigma_{n-1}$, or equivalently,
$\sigma_{n-1}^{-1} \le (c+1) \sigma_n^{-1}$, and obtain from
(\ref{f419}) the estimate
\[
\| x_n^\delta - x^\dagger \|_p \le \frac{E}{C-1} \cdot \sqrt{c+1} +
E \cdot M/m .
\]
From the both cases we have that $\| x_n^\delta - x^\dagger\|_p$ can
be estimated by
\begin{equation}\label{f420}
\| x_n^\delta - x^\dagger \|_p \le \frac{E}{C-1} \cdot \max \left \{
M/m, \sqrt{c+1} \right \} + E \cdot M/m .
\end{equation}
In the {\it second step} we proceed according to the proof of
(\ref{f417}) and obtain
\begin{equation}\label{f421}
\| \varrho (G) (x_n^\delta - x^\dagger ) \| \le (C+1) \delta/m .
\end{equation}
In the final {\it third step} of the proof we use the both estimates
(\ref{f420}) and (\ref{f421}), apply Proposition \ref{p27} and
obtain (\ref{f418}).
\end{proof}

\subsection{Discrepancy principle revisited}
The error bounds given in Subsection 4.2 require in both cases of
high order and low order regularization the both link conditions
A1(i) and A1(ii), and the assumption $C_1 > 1$ in the discrepancy
principle (\ref{f42}). We will show in this subsection that in the
case of low order regularization $s \le p$ order optimal error
bounds can be obtained without the second link condition A1(ii). Our
estimate in Theorem \ref{t48} shows that $C_1 = C_2 = 1$ in the
discrepancy principle (\ref{f42}) is best possible in the sense of
minimal error bounds.
We start our study with some important inequality.
\begin{proposition}\label{p47}
For $0 \le s \le p$, the regularized solution $x_n^\delta$ defined
by (\ref{f21}) obeys the estimate
\begin{eqnarray}\label{f422}
\| A x_n^\delta - y^\delta \|^2 + \sigma_n^{-1} \| x_n^\delta -
x^\dagger \|_s^2 & \le &  \sigma_n^{-1} \| r_n^{1/2} (T^*T)  G^{-s}
( x^\dagger - x_0
) \|^2\nonumber \\[1ex]
& & + \, \| y-y^\delta\|^2  .
\end{eqnarray}
\end{proposition}
\begin{proof}
Let $A: X_s \to Y$ be the restriction of $A$ to $X_s \subset X$ and
$A_s^*: Y \to X_s$ its adjoint. Due to the valid identity
$ (Ax,y) = (x, A_s^* y)_s = (x,A^*y) = (x,G^{2s} A^*y)_s $
for all $x \in X_s$ and $y \in Y$ we conclude that the adjoint
$A_s^*: Y \to X_s$ of the operator $A: X_s \to Y$ is given by
$ A_s^* = G^{2s} A^* $.
The operator $A_s^*A: X_s \to X_s$ is self-adjoint. Further, there
holds
\begin{equation}\label{f423}
G^s g_n (T^*T) = g_n (A_s^*A) G^{s} .
\end{equation}
Consequently, the regularized solution (\ref{f21}) which is an
element of the space $X_s$ can be written in the equivalent form
\[
x_n^\delta - x_0 = g_n (A_s^* A) A_s^* ( y^\delta - A x_0)  .
\]
From the valid identity
$ x_n^\delta -x^\dagger  = - r_n(A^*_sA) ( x^\dagger - x_0) + g_n
(A_s^*A)A_s^*(y^\delta - A x^\dagger) $
and the identity $ g_n (A_s^*A)A_s^*=A_s^*g_n (AA_s^*) $ we obtain
\begin{eqnarray}\label{f424}
\|x_n^\delta -x^\dagger\|^2_s & = & \|r_n(A^*_sA) ( x^\dagger - x_0)
\|^2_s +
\|g_n (A_s^*A)A_s^*(y^\delta - Ax^\dagger)\|^2_s \nonumber \\[1ex]
&& - 2\big(Ag_n (A_s^*A) r_n(A_s^*A) ( x^\dagger - x_0)  , y^\delta
-Ax^\dagger \big)  .
\end{eqnarray}
We introduce the abbreviations
\[
R_n :=g_n(AA_s^*) r_n(AA_s^*)  \quad \mbox{and} \quad y_0^\delta :=
y^\delta - A x_0 ,
\]
decompose $y_0^\delta$ into the sum $A(x^\dagger - x_0) $ plus
$y^\delta - A  x^\dagger  $  and obtain the equality
\begin{eqnarray}\label{f425}
 (R_n y_0^\delta , y_0^\delta) & = &
 (R_n A (x^\dagger - x_0 ) , A (x^\dagger - x_0) )  +
 (R_n (Ax^\dagger-y^\delta) ,  Ax^\dagger- y^\delta) \nonumber
 \\[1ex]
  & & + \, 2 ( R_n A (x^\dagger - x_0) , y^\delta - Ax^\dagger)   .
\end{eqnarray}
Addition of the equations (\ref{f424}) and (\ref{f425}) yields
\begin{eqnarray}
(R_n y_0^\delta, y_0^\delta) +  \|x_n^\delta -x^\dagger\|^2_s & = &
 \|r_n(A^*_sA) ( x^\dagger - x_0)  \|^2_s +
 \|g_n (A_s^*A)A_s^*(y^\delta - Ax^\dagger )\|^2_s \nonumber \\[1ex]
 &&   + \,
 (A_s^*R_n A ( x^\dagger - x_0) , ( x^\dagger - x_0) )_s  \nonumber \\[1ex]
 & & + \, (R_n (Ax^\dagger-y^\delta)  ,  Ax^\dagger-y^\delta )
 . \nonumber
\end{eqnarray}
We use the valid identities
\[
 r_n^2 (A_s^*A) \, + \, A_s^*R_n A \,
  = \, r_n(A_s^*A), \quad
 A g_n^2(A_s^*A)A_s^* \, + \, R_n \, = \, g_n(AA_s^*)
\]
and obtain from the above equation
\begin{eqnarray}\label{f426}
(R_n y_0^\delta  , y_0^\delta)  + \|x_n^\delta -x^\dagger\|^2_s & =
& (r_n (A_s^*A) ( x^\dagger - x_0) , x^\dagger - x_0)_s \nonumber
\\[1ex]
& &  + \, (g_n (AA_s^*)(Ax^\dagger-y^\delta)  ,  Ax^\dagger-y^\delta
) .
\end{eqnarray}
By exploiting properties (i) and (iv) of Proposition \ref{p21}, we
obtain
\[
\begin{array}{ll}
 \mbox{(a) } & \sigma_n^{-1} (g_n (AA_s^*)(Ax^\dagger-y^\delta)  ,
 Ax^\dagger-y^\delta )
  \le  \|Ax^\dagger-y^\delta\|^2 , \\[1ex]
 \mbox{(b) } & \sigma_n^{-1} (R_n y_0^\delta , y_0^\delta )
  \ge  \left ( r_n^2 (AA_s^*) y_0^\delta
 , y_0^\delta \right )
  =  \left \| r_n(AA_s^*) y_0^\delta \right \|^2
  =    \| Ax_n^\delta - y^\delta \|^2  .
\end{array}
\]
We multiply (\ref{f426}) by $\sigma_n^{-1}$, use the estimates (a)
and (b) and obtain
\begin{equation}\label{f427}
\|Ax_n^\delta -y^\delta\|^2  + \sigma_n^{-1} \|x_n^\delta
-x^\dagger\|^2_s \le \|Ax^\dagger-y^\delta\|^2 + \sigma_n^{-1} (r_n
(A_s^*A) ( x^\dagger - x_0) , x^\dagger - x_0 )_s .
\end{equation}
Finally we observe that due to (\ref{f423}) we have
\[
 (r_n (A_s^*A) ( x^\dagger - x_0) , x^\dagger - x_0 )_s = \|r_n^{1/2} (T^*T)
 G^{-s}
( x^\dagger - x_0) \|^2 .
\]
From this identity and (\ref{f427}) we obtain (\ref{f422}).
\end{proof}
From Proposition \ref{p47} and Proposition \ref{p28} we obtain the
main result of this subsection.

\begin{theorem}\label{t48}
Let $x_n^\delta$ be defined by (\ref{f21}) and $\sigma_n$ be chosen
by the discrepancy principle (\ref{f42}) with $1 \le C_1 \le C_2$,
assume the link condition A1(i), the solution smoothness A2 and $0
\le s \le p$.
If $f$ defined by (\ref{f220}) is convex, and $\xi_s (t):= \psi_s^2
(t^{1/(2s)})$ is convex where $\psi_s$ is given by (\ref{f212}),
then
\begin{equation}\label{f428}
\| x_n^\delta - x^\dagger\| \le  E  \left [ \psi_p^{-1} \left (
\frac{(C_2 +1) \delta}{m E} \right ) \right ]^p .
\end{equation}
\end{theorem}
\begin{proof}
For $\sigma_n$ chosen by the discrepancy principle (\ref{f42}) the
estimate (\ref{f422}) of Proposition \ref{p47} attains the form
\[
C_1^2 \delta^2 + \sigma_n^{-1} \| x_n^\delta - x^\dagger \|_s^2 \le
\delta^2 + \sigma_n^{-1} \| r_n^{1/2} (T^*T)  G^{-s} ( x^\dagger -
x_0 ) \|^2 .
\]
Since $C_1 \ge 1$, we have $\| x_n^\delta - x^\dagger \|_s^2 \le \|
r_n^{1/2} (T^*T)  G^{-s} ( x^\dagger - x_0) \|^2$. We use the
representation $G^{-s} (x^\dagger - x_n) = r_n (T^*T) G^{-s} (
x^\dagger - x_0) $, see (\ref{f26}), use Assumption A2 and obtain
\begin{eqnarray}\label{f429}
\| x_n^\delta - x^\dagger \|_s^2 & \le & \| r_n^{1/2} (T^*T) G^{-s}
( x^\dagger - x_0) \|^2 \nonumber \\
& = & ( G^{p-2s} (x^\dagger - x_n), G^{-p} ( x^\dagger - x_0) ) \nonumber \\
& \le & E \| x_n - x^\dagger \|_{2s-p} ,
\end{eqnarray}
where $x_n$ is the regularized solution with exact data. For
estimating $\| x_n - x^\dagger \|_{2s-p}$, we use estimate
(\ref{f221}) of Proposition \ref{p28} and obtain
\begin{equation}\label{f430}
\| x_n^\delta - x^\dagger \|_s \le E \left [ \psi_p^{-1} \left (
\frac{\| \varrho (G) (x_n-x^\dagger )\|}{E } \right ) \right ]^{p-s}
\,.
\end{equation}
For estimating $\| Ax_n-Ax^\dagger \|$, we use (\ref{f28}), the
identity $r_n (TT^*) ( y^\delta - A x_0) = y^\delta - Ax_n^\delta$,
$r_n (\lambda) \le 1$ and (\ref{f42}) and obtain the estimate
\begin{eqnarray}
\| Ax_n-Ax^\dagger \| & = & \| r_n (TT^*) ( y - A x_0) \| \nonumber
\\[1ex]
& \le &  \| r_n (TT^*) ( y^\delta - A x_0) \| + \| r_n (TT^*)
(y-y^\delta )\| \nonumber \\[1ex]
& \le & (C_2 +1) \delta. \nonumber
\end{eqnarray}
Hence, by using A1(i) we have
$ \| \varrho (G) (x_n-x^\dagger )\| \le (C_2 +1) \delta /m $.
Since $\psi_p^{-1}$ is monotone, we obtain from (\ref{f430}) the
estimate
\begin{equation}\label{f431}
\| x_n^\delta - x^\dagger \|_s \le E \left [ \psi_p^{-1} \left (
\frac{(C_2 +1) \delta}{m E } \right ) \right ]^{p-s} \,.
\end{equation}
Next, let us estimate $\| \varrho (G) ( x_n^\delta - x^\dagger) \|
$. Using Assumption A1(i) and the estimate $\| Ax_n^\delta -
Ax^\dagger \| \le \| A x_n^\delta - y^\delta \| + \| y-y^\delta\|
\le (C_2 +1) \delta $ yields
\begin{equation}\label{f432}
\| \varrho (G) ( x_n^\delta - x^\dagger )  \| \le  (C_2 +1) \delta /
m .
\end{equation}
Now we apply the interpolation estimate (\ref{f219}) of Proposition
\ref{p27} and obtain by using (\ref{f431}), (\ref{f432}) and the
abbreviation $\delta_1:= \frac{(C_2 +1) \delta}{mE}$ that
\begin{equation}\label{f433}
\| x_n^\delta - x^\dagger \| \le E \left [ \psi_p^{-1} (\delta_1 )
\right ]^{p-s} \cdot \left [ \psi_s^{-1} \left ( \delta_1 [
\psi_p^{-1} (\delta_1) ]^{s-p} \right ) \right ]^s .
\end{equation}
From the first equation of (\ref{f324}) we have $\psi_s^{-1} \left (
\psi_p (\lambda) \cdot \lambda^{s-p} \right ) = \lambda$.
Substituting $\lambda = \psi_p^{-1} ( \delta_1 )$ yields
$\psi_s^{-1} \left (  \delta_1  [\psi_p^{-1} ( \delta_1 )]^{s-p}
\right ) = \psi_p^{-1} ( \delta_1 )$. From this equation and
(\ref{f433}) we obtain (\ref{f428}).
\end{proof}
%

\section{Practical implementation}
\setcounter{equation}{0}
For the practical application of implicit iteration methods in
Hilbert scales one has to make different decisions: First, one has
to choose the operator $B$,  second, one has to fix the number $s$
in the method (\ref{f12}), third, one has to choose the starting
value $x_0$ and to fix the numbers $\alpha_k$, $k=1,...,n$, and
fourth, one has effectively to realize the discrepancy principle
(\ref{f41}) with a little number $n$ of iteration steps. The choice
of $B$ and $x_0$ should be done in dependence on the expected
smoothness of the element $x^\dagger - x_0$ such that Assumption A2
holds true for $p$ sufficiently large, and $s$ should have the
magnitude of $p$. In our further study we concentrate on the choice
of the numbers $\alpha_k$, $k=1,...,n$ for effectively realizing the
discrepancy principle (\ref{f41}) or (\ref{f42}) or (\ref{f43}),
respectively, with a little number $n$ of iteration steps. In a
first proposition we give an upper bound for the regularization
parameter of the discrepancy principle in case $n=1$ which will
serve us as starting value for the iteration (1.2). To our best
knowledge, so far there have not been upper bounds for the
regularization parameter of the discrepancy principle in the
literature.
\begin{proposition}\label{p51}
Let $n=1$, let $x_1^\delta$ the regularized solution (\ref{f21}) and
let $\alpha_1=\alpha_D$ be chosen by the discrepancy principle
(\ref{f41}) with $C \ge 1$. If $\| y^\delta - A x_0\| > C \delta$,
then
\begin{equation}\label{f51}
\alpha_D < \frac{ C \delta \, \| G^s A^* ( y^\delta - A x_0)
\|^2}{\left ( \| y^\delta- A x_0 \| - C \delta \right ) \| y^\delta
- A x_0 \|^2} \, .
\end{equation}
\end{proposition}
\begin{proof} Let $x_\alpha^\delta = x_0 - (A^*A + \alpha G^{-2s} )^{-1} A^* (Ax_0 - y^\delta
)$ and $\alpha = \alpha_D$ be the regularization parameter that
obeys the discrepancy principle $\| A x_\alpha^\delta - y^\delta \|
= C \delta$. For solving this nonlinear equation, Newton's method
applied to the equivalent equation
\begin{equation}\label{f52}
g(r) = \| A x_{1/r}^\delta - y^\delta \|^{-1} - (C\delta)^{-1} = 0
\end{equation}
is studied in \cite{lpst10} which results in the iteration
\begin{equation}\label{f53}
r_{k+1} = r_k - \frac{\| A x_{1/r_k}^\delta - y^\delta \|^{-1} -
(C\delta)^{-1}}{r_k^{-3} \left ( v_{1/r_k}^\delta , G^{-2s} (
x_{1/r_k}^\delta - x_0) \right ) \, \| A x_{1/r_k}^\delta - y^\delta
\|^{-3} }
\end{equation}
where $v_{1/r}^\delta $ is given by $v_{1/r}^\delta = (A^*A + r^{-1}
G^{-2s} )^{-1} G^{-2s} (x_{1/r}^\delta - x_0)$. From \cite[Theorem
3.5]{lpst10} we know that the iteration (\ref{f53}) possesses the
following properties:
\begin{enumerate}
\item[(i)] The sequence $(r_k)$ converges globally and monotonically from the
left to $r_D$ for any starting values $0 \le r_0 < r_D:=
1/\alpha_D$.
\item[(ii)] The speed of convergence is locally quadratic.
\end{enumerate}
For $r \to 0$, the both limit relations
\[
\lim_{r \to 0+0} x_{1/r}^\delta = x_0 \enspace \mbox{and} \enspace
\lim_{r \to 0+0} r_k^{-3} \left ( v_{1/r}^\delta , G^{-2s} (
x_{1/r}^\delta - x_0) \right ) = \| G^s A^* ( y^\delta - A x_0) \|^2
\]
are valid. We execute one iteration step of the iteration
(\ref{f53}) with starting value $r_0 = 0$ and obtain due to the
above limit relations that
\[
r_1 = \frac{\left ( \| y^\delta- A x_0 \| - C \delta \right ) \|
y^\delta - A x_0 \|^2 }{C \delta \, \| G^s A^* ( y^\delta - A x_0)
\|^2 } \, .
\]
Due to the above property (i) we have $r_1 < r_D$. Since $r$ and
$\alpha$ are related by $\alpha = 1/r$ we obtain (\ref{f51}).
\end{proof}

Based on the Newton iteration (\ref{f53}) we propose following
strategy for effectively realizing the discrepancy principle
(\ref{f43}) with a little number $n$ of iteration steps.

\medskip

\hrule \hrule  \vspace{1ex} {\parindent 0pt \bf Algorithm 1} $\,$
Global convergent Newton iteration for rule (\ref{f43}) \\[-1.4ex]
\hrule \vspace{1.2ex}
\begin{itemize}
\item[1: ] \hspace{-0.2cm} Start with initial data
$y^\delta$, $A$, $G$, $s$,  $\delta$, $C:=1.1$ and $x_0$.
\\[-2ex]
\item[2: ] \hspace{-0.2cm} {\bf if} $\, \| A x_0 - y^\delta \| > C \delta\,$ {\bf then}  \\[-2ex]
\item[3: ] \hspace{0.2cm} Compute $\alpha$ by the right hand side of (\ref{f51}) with $C=1$.
\\[-2ex]
\item[4: ] \hspace{0.2cm} Compute $x:= x_0 - (A^*A + \alpha G^{-2s} )^{-1} A^* (A x_0 - y^\delta)$
and set $n:=1$. \\[-2ex]
\item[5: ] \hspace{0.2cm} {\bf while} $\,\| A x - y^\delta \| > C \delta\,$ {\bf do}   \\[-2ex]
\item[6: ] \hspace{0.6cm} Compute $v:= (A^*A + \alpha G^{-2s} )^{-1} G^{-2s} (x-x_0)$.   \\[-1ex]
\item[7: ] \hspace{0.6cm} Update
$\displaystyle{r:= \frac{1}{\alpha} - \frac{\| Ax - y^\delta\|^{-1} - \delta^{-1}}
{\alpha^3 \left ( v, G^{-2s} (x-x_0) \right ) \, \| Ax-y^\delta\|^{-3}} }$, $n:=n+1$, \\[-0ex]
\item[8: ] \hspace{0.6cm} $x_0:=x$, $\alpha:=1/r$,
$x:= x_0 - (A^*A + \alpha G^{-2s} )^{-1} A^* (A x_0 - y^\delta)$.  \\[-2ex]
\item[9: ] \hspace{0.2cm} {\bf end while} \\[-2ex]
\item[10: ] \hspace{-0.2cm} {\bf end if} \\[-1ex]
\hrule
\end{itemize}

\medskip

For discussing some properties of Algorithm 1, we will work with the
notation
\[
x_k^\delta (\alpha) := x_{k-1}^\delta - (A^*A + \alpha B^{2s} )^{-1}
A^* (
 A x_{k-1}^\delta - y^\delta ), \quad k=1,2,...,
\]
that indicates the dependence of $x_k^\delta$ defined by (\ref{f12})
on the parameter $\alpha$. We start by some monotonicity property of
the sequence $(\alpha_k)_{k=1}^n$ in the iteration (\ref{f12}).
\begin{proposition}\label{p52}
The regularized solutions $x_k^\delta$, $k=1, \dots ,n$, obtained by
Algorithm 1 have the form (\ref{f12}). The related sequence
$(\alpha_k)_{k=1}^n$ is strictly monotonically decreasing.
\end{proposition}
\begin{proof}
In steps 3 and 4 of Algorithm 1, $\alpha_1$ and $x_1^\delta =
x_1^\delta (\alpha_1)$ are computed. Then, the while loop (steps 5
-- 9 of Algorithm 1) is executed $n-1$ times to obtain $\alpha_k$
and $x_k^\delta = x_k^\delta (\alpha_k)$ for $k=2,...,n$. The
parameter $\alpha = \alpha_{k} := 1/r_{k}$ (see step 7 of Algorithm
1) is obtained by performing one Newton step for solving the
nonlinear equation
\[
g(r)= \| A x_{k-1}^\delta (1/r) - y^\delta \|^{-1} - \delta^{-1} = 0
\]
with starting value $r_{k-1} = 1/\alpha_{k-1}$.
It can be shown (compare \cite{lpst10}) that the function $g$
possesses following properties:
\begin{enumerate}
\item[(i)] There hold the two limit relations
\[
\lim_{r \to 0 + 0} g(r) = \| A x_{k-2}^\delta - y^\delta \|^{-1} -
\delta^{-1} < 0 \quad \mbox{and} \quad \lim_{r \to \infty} g(r) = +
\infty .
\]
\item[(ii)] The function $g: \mathbb R^+ \to \mathbb R$
is monotonically increasing and concave.
\end{enumerate}
From these properties and $g(r_{k-1}) < 0$ we conclude that $r_k >
r_{k-1}$. It follows that $\alpha_{k} < \alpha_{k-1}$ for all $k=2,
\dots , n$.
\end{proof}

For discussing convergence properties of Algorithm 1 we consider
Tikhonov regularization
\begin{equation}\label{f54}
x_1^\delta (\beta_m ): = x_0 - (A^*A + \beta_m G^{-2s} )^{-1} A^*
(Ax_0 - y^\delta )
\end{equation}
and assume
\begin{enumerate}
\item[(i)] $\beta_k = r_k^{-1}$, $k=2, \dots, m$, is obtained by the iteration
\begin{equation}\label{f55}
r_{k} = r_{k-1} - \frac{\| A x_{1/r_{k-1}}^\delta - y^\delta \|^{-1}
- \delta^{-1}}{r_{k-1}^{-3} \left ( v_{1/r_{k-1}}^\delta , G^{-2s} (
x_{1/r_{k-1}}^\delta - x_0) \right ) \, \| A x_{1/r_{k-1}}^\delta -
y^\delta \|^{-3} },
\end{equation}
where  $v_{1/r}^\delta $ is given by $v_{1/r}^\delta = (A^*A +
r^{-1} G^{-2s} )^{-1} G^{-2s} (x_{1/r}^\delta - x_0)$ and
$x_{1/r}^\delta$ is given by $x_{1/r}^\delta = x_1^\delta (1/r)$,
\item[(ii)] $r_1$ is chosen as $\displaystyle{r_1 = \frac{ ( \| y^\delta- A x_0 \| - \delta ) \|
y^\delta - A x_0 \|^2 }{  \delta \, \| G^s A^* ( y^\delta - A x_0)
\|^2  }}$
and the iteration (\ref{f55}) is stopped with the first integer $m$
for which, with $C:=1.1$,
\begin{equation}\label{f56}
\| A x_1^\delta (\beta_m) - y^\delta \| \le C \delta < \| A
x_1^\delta (\beta_k) - y^\delta \| , \quad 0 \le k < m .
\end{equation}
\end{enumerate}

From \cite{lpst10} we know that the iteration (\ref{f55}) converges
globally and monotonically from the left to the solution of the
equation $ g(r) = \| A x_1^\delta (1/r) - y^\delta \|^{-1} -
\delta^{-1} = 0$, and that in the vicinity of the solution we have
quadratic speed of convergence. It follows that by the stopping rule
(\ref{f56}) a finite number $m$ of iteration steps is defined. Our
next proposition tells us that Algorithm 1 is not slower than the
iteration (\ref{f55}) with stopping rule (\ref{f56}).

\begin{proposition}\label{p53}
Let $m$ be the number of iterations of method (\ref{f55}) with
stopping rule (\ref{f56}). Then, $n \le m$, where $n$ is the number
of iterations of Algorithm 1.
\end{proposition}
\begin{proof}
Assume that $\alpha_1$ and $x_1^\delta (\alpha_1)$ in steps 3 and 4
of Algorithm 1 are computed, which coincide with $\beta_1$ and
$x_1^\delta (\beta_1)$ of the iteration (\ref{f55}). Then, in the
first iteration step of the while-loop (steps 5 -- 9 of Algorithm 1)
we obtain $\alpha_2$ and
\[
x_2^\delta (\alpha_2) = x_1^\delta (\alpha_1) - (A^*A + \alpha_2
G^{-2s} )^{-1} A^* (A x_1^\delta (\alpha_1) - y^\delta) .
\]
For $x_2^\delta (\alpha_2)$ computed in this way we have
\begin{equation}\label{f57}
y^\delta - A x_2^\delta (\alpha_2)= \alpha_2 (TT^* + \alpha_2 I
)^{-1} \alpha_1 (TT^* + \alpha_1 I)^{-1} (y^\delta - A x_0) .
\end{equation}
On the other hand, from the iteration (\ref{f55}) we obtain after
the first step the regularization parameter $\beta_2 = \alpha_2$ and
the regularized solution $x_1^\delta (\beta_2)$ which obeys
\begin{equation}\label{f58}
y^\delta - A x_1^\delta (\beta_2) = \beta_2 (TT^* + \beta_2 )^{-1}
(y^\delta - A x_0) .
\end{equation}
Comparing both identities (\ref{f57}) and (\ref{f58}) and observing
that $\alpha_2 = \beta_2$ we obtain that
$\|y^\delta - A x_2^\delta (\alpha_2) \| < \| y^\delta - A
x_1^\delta (\beta_2) \| $.
In a similar way we obtain that
\[
\|y^\delta - A x_k^\delta (\alpha_k) \| < \| y^\delta - A x_1^\delta
(\beta_k) \| , \quad k = 3, \dots , n ,
\]
where $(\alpha_k)$ is the sequence generated by Algorithm 1 and
$(\beta_k)$ is the sequence generated by (\ref{f55}). From this
estimate we obtain that Algorithm 1 terminates not later than the
iteration (\ref{f55}) with stopping rule (\ref{f56}).
\end{proof}

After termination of Algorithm 1, different cases can appear:
\begin{enumerate}
\item We have $\delta < \| A x_n^\delta - y^\delta \| \le C \delta$ with $C=1.1$. In this
case, all three Theorems \ref{t45}, \ref{t46} and \ref{t48} apply.
\item We have $\| A x_n^\delta - y^\delta \| =  \delta $.
Then, both Theorems \ref{t46} and \ref{t48} apply.
\item We have $\| A x_n^\delta - y^\delta \| < \delta $. In this case,
Theorem \ref{t46} applies.
\end{enumerate}

Our next proposition tells us that in all three termination cases
(1) -- (3), the additional assumption (\ref{f44}) of Theorem
\ref{t46} is satisfied with some $c < 1$.

\begin{proposition} \label{p54}
The regularized solution $x_n^\delta$ obtained by Algorithm 1 has
the form (\ref{f21}) with some sequence $(\alpha_k)_{k=1}^n$ that
obeys assumption (\ref{f44}) with $c<1$.
\end{proposition}
\begin{proof}
Consider the final iteration of the while-loop (steps 5 -- 9 of
Algorithm 1). Starting from $x = x_{n-1}^\delta (\alpha_{n-1} )$
with $\| Ax - y^\delta\|> C \delta$, $\alpha_n:=1/r_n$ is obtained
by performing one Newton step for solving the nonlinear equation
\[
g(r)= \| A x_{n-1}^\delta (1/r) - y^\delta \|^{-1} - \delta^{-1} = 0
\]
with starting value $r_{n-1} = 1/\alpha_{n-1}$. As a result, we
obtain some $\alpha_n < \alpha_{n-1}$, see Proposition \ref{p52},
and the final regularized solution $x_n^\delta$ is obtained by
\[
x_n^\delta (\alpha_n) = x_{n-1}^\delta (\alpha_{n-1}) - (A^*A +
\alpha_n G^{-2s} )^{-1} A^* (A x_{n-1}^\delta (\alpha_{n-1}) -
y^\delta) .
\]
Some formal computations show that $x_n^\delta (\alpha_n)$ can be
rewritten as
\[
x_n^\delta (\alpha_n) = x_{n-1}^\delta (\alpha_{n}) - (A^*A +
\alpha_{n-1} G^{-2s} )^{-1} A^* (A x_{n-1}^\delta (\alpha_{n}) -
y^\delta) .
\]
Since the function $g$ is monotonically increasing and concave and
since $g(r_{n-1}) < 0$
we conclude that the element $x_{n-1}^\delta (\alpha_{n})$ obeys
$ \| A x_{n-1}^\delta (\alpha_{n}) - y^\delta \| > \delta $.
It follows that the final two parameters $\alpha_{n-1}$ and
$\alpha_n$ in the iteration (\ref{f12}) can be interchanged such
that we have $\alpha_n > \alpha_{n-1}$. This yields (\ref{f44}) with
some constant $c < 1$.
\end{proof}

\section{Numerical experiments}
\setcounter{equation}{0}
In this section we perform numerical experiments for computing
regularized solutions by Algorithm 1. We consider Fredholm integral
equations
\begin{equation}\label{f61}
[A x] (s):= \int_0^1 K (s,t) x(t) \, \mbox{d} t = y(s) , \enspace 0
\le s \le 1 , \enspace A: L^2 (0,1) \to L^2 (0,1)
\end{equation}
and differential operators $B: D \subset L^2 (0,1) \to L^2 (0,1)$ of
first order defined by
\begin{equation}\label{f62}
Bx = \sum_{k=1}^\infty k (x,e_k) e_k  \quad \mbox{with} \quad e_k
(t) = \sqrt{2} \sin \left ( k \pi t  \right ) .
\end{equation}

\begin{example}\label{e61} Our test example ({\it deriv2} from \cite{ha94}) is
(\ref {f61}) with kernel function
\[
K(s,t) = \left \{
\begin{array}{ll}
s (1-t) \quad & \mbox{for $ s \le t$} \\[1ex]
t (1-s) \quad & \mbox{for $ s \ge t$} .
\end{array}
\right.
\]
For this kernel function, Assumption A1 is satisfied with
$m=M=\pi^{-2}$ and $\varrho (t) = t^2$. We consider three
subexamples in which the right hand sides $y(s)$, the corresponding
solutions $x^\dagger (t)$ and the maximal smoothness parameters
$p_0$ for which Assumption A2 with $x_0 = 0$ holds true for all $p
\in (0, p_0)$, are given by
\begin{equation*}
\begin{array}{llll}
\mbox{(i)} & y(s) = - \frac{1}{4 \pi^2} \sin 2 \pi s , & x^\dagger
(t) = \sin 2 \pi t, & p_0 = \infty,
\\[1ex]
\mbox{(ii)} & y(s) = \frac{s}{3} \left ( 1 - 2s^2 + s^3 \right ), &
x^\dagger (t) = 4t (1-t), & p_0 = \frac{5}{2},
\\[1ex]
\mbox{(iii)} & y(s) = \frac{s}{6} \left ( 1 - s^2 \right ), &
x^\dagger (t) = t, & p_0 = \frac{1}{2}.
\end{array}
\end{equation*}
\end{example}

The discretization of (\ref{f61}) has been done by Galerkin
approximation as outlined, e.\,g., in \cite{ha94,lpst10},
guaranteeing that $\| x^\dagger \|_2 \approx \| x^\dagger (t)
\|_{L^2 (0,1)}$ and $\| y\|_2 \approx \| y(s) \|_{L^2 (0,1)}$ holds.
As a discrete approximation of the first order differential operator
(\ref{f62}) we use the $(m,m)$ -- matrix
\begin{eqnarray}\label{f63}
B:= B_2^{1/2} \quad \mbox{with} \quad B_2 = \frac{(m+1)^2}{\pi^2}
\left(
\begin{array}{rrrr}
2 & -1 & &  \\
-1 & \ddots  & \ddots & \\
& \ddots & \ddots & -1 \\
&  & - 1 & 2
\end{array}
\right)  ,
\end{eqnarray}
compare \cite{lpst10}. For modeling noise in the discretized right
hand side $y \in \mathbb R^m$, for given nonnegative $\sigma$ we
compute
\[
y^\delta = y + \sigma \, \frac{\| y\|_2}{\|e\|_2} \, e ,
\]
where $e= (e_i)$ is a random vector with $e_i \sim {\mathcal N}
(0,1)$. In this way of modeling noise we guarantee that for the
relative error we have $\| y - y^\delta \|_2 / \| y\|_2 = \sigma$.
The noise level $\delta$ is then given by
$ \delta = \sigma \, \| y \|_2$.
Tables 1 -- 3 show our numerical results with $x_0 = 0$, where the
letter codes in the leftmost column refer to following three
iteration methods:
\begin{itemize}
\item {\bf (TI/DP):} This is the method of Tikhonov regularization (\ref{f54}) with $x_0 = 0$ and $s=1$,
where the regularization parameter obeys (\ref{f56}) and is obtained
by the iteration (\ref{f55}) which converges globally and locally
quadratically.
\item {\bf (IIM/A1):} This is the implicit iteration method
(\ref{f12}) with  $x_0 = 0$ and $s=1$, where the sequence
$(\alpha_k)_{k=1}^n$ is obtained by Algorithm 1.
\item {\bf (IIM/GS):} This is the implicit iteration method
(\ref{f12}) with  $x_0 = 0$ and $s=1$, where the sequence
$(\alpha_k)_{k=1}^n$ is the geometric sequence $(q^{k-1}
\alpha_1)_{k=1}^n$  with $q = \frac{1}{2}$ as proposed in
\cite{hg98} and stopping rule (\ref{f43}) with $C=1.1$.
\end{itemize}

For all three iteration methods our tables contain
\begin{enumerate}
\item[(i)] the number $n$ of required iterations,
\item[(ii)] the final regularization parameter $\alpha_n$,
\item[(iii)] the discrepancy $d_n := \|A x_1^\delta (\alpha_n) - y^\delta\|_2$ of the final approximation
for the iteration method (TI/DP), and the discrepancy $d_n := \|A
x_n^\delta - y^\delta\|_2$ of the final approximation for the
iteration methods (IIM/A1) and (IIM/GS),
\item[(iv)] the error $e_n := \|x_1^\delta (\alpha_n) - x^\dagger\|_2$ of the final approximation
for the iteration method (TI/DP), and the error $e_n := \|x_n^\delta
- x^\dagger\|_2$ of the final approximation for the iteration
methods (IIM/A1) and (IIM/GS).
\end{enumerate}

In our experiments, all three iteration methods have been started
first with
\begin{equation}\label{f64}
\alpha_1 = \delta \, \frac{( B^{-2} A^* y^\delta, A^* y^\delta )
}{\left ( \| y^\delta\|_2 - \delta \right ) \| y^\delta\|_2^2 } ,
\end{equation}
compare (\ref{f51}), and second with $\alpha_1 = 1$ as done in
\cite{hg98}. In order to keep the discretization error small, we
have used the dimension number $m=400$ in all computations.

\begin{table}[ht]
\begin{tabular}{ccccccc}
\hline\\[-0.52cm]
\hline \\[-0.35cm]
$\quad \quad {\rm Method} \quad \quad$ & $\quad n \quad$ & $\quad
\alpha_n \quad$ & $\quad d_n \quad$ & $\quad e_n \quad$  \\[0.1cm]
\hline \\[-0.35cm]
(TI/DP) & $\quad 3 \quad$ & $\quad 5.54$ E$-7\quad$ & $\quad 1.88$ E$-4\quad$ & $\quad 3.79$ E$-3\quad$ \\
(IIM/A1) & $\quad 2 \quad$ & $\quad 8.85$ E$-7\quad$ & $\quad 1.78$ E$-4\quad$ & $\quad 2.94$ E$-3\quad$ \\
(IIM/GS) & $\quad 2 \quad$ & $\quad 8.10$ E$-7\quad$ & $1.78$ E$-4$ & $\quad 2.95$ E$-3\quad$ \\[0.1cm]
\hline\\[-0.35cm]
%
%
(TI/DP) & $\quad 4 \quad$ & $\quad 5.54$ E$-7\quad$ & $\quad 1.88$ E$-4\quad$ & $\quad 3.79$ E$-3\quad$ \\
(IIM/A1) & $\quad 3 \quad$ & $\quad 8.85$ E$-7\quad$ & $\quad 1.78$ E$-4\quad$ & $\quad 2.94$ E$-3\quad$ \\
(IIM/GS) & $\quad 17 \quad$ & $\quad 1.52$ E$-5\quad$ & $1.78$ E$-4$ & $\quad 2.83$ E$-3\quad$ \\[0.1cm]
\hline\\[-0.52cm]
\hline
\end{tabular}
\caption{Example \ref{e61} (i) with $\sigma = 0.01$ $(\delta =
\sigma \| y\|_2 \approx 1.79$ E$-4)$. {\it Top}: $\alpha_1$ from
(\ref{f64}), {\it Down}: $\alpha_1 = 1$.} \label{tab1}
\end{table}

\begin{table}[ht]
\begin{tabular}{ccccccc}
\hline\\[-0.52cm]
\hline \\[-0.35cm]
$\quad \quad {\rm Method} \quad \quad$ & $\quad n \quad$ & $\quad
\alpha_n \quad$ & $\quad d_n \quad$ & $\quad e_n \quad$  \\[0.1cm]
\hline \\[-0.35cm]
(TI/DP) & $\quad 3 \quad$ & $\quad 2.75$ E$-5\quad$ & $\quad 7.97$ E$-4\quad$ & $\quad 1.62$ E$-2\quad$ \\
(IIM/A1) & $\quad 2 \quad$ & $\quad 5.14$ E$-5\quad$ & $\quad 7.75$ E$-4\quad$ & $\quad 1.70$ E$-2\quad$ \\
(IIM/GS) & $\quad 2 \quad$ & $\quad 5.18$ E$-5\quad$ & $7.75$ E$-4$ & $\quad 1.70$ E$-2\quad$ \\[0.1cm]
\hline\\[-0.35cm]
%
%
(TI/DP) & $\quad 4 \quad$ & $\quad 2.75$ E$-5\quad$ & $\quad 7.97$ E$-4\quad$ & $\quad 1.62$ E$-2\quad$ \\
(IIM/A1) & $\quad 3 \quad$ & $\quad 5.15$ E$-5\quad$ & $\quad 7.75$ E$-4\quad$ & $\quad 1.70$ E$-2\quad$ \\
(IIM/GS) & $\quad 12 \quad$ & $\quad 4.88$ E$-4\quad$ & $8.08$ E$-4$ & $\quad 2.46$ E$-2\quad$ \\[0.1cm]
\hline\\[-0.52cm]
\hline
\end{tabular}
\caption{Example \ref{e61} (ii) with $\sigma = 0.01$ $(\delta =
\sigma \| y\|_2 \approx 7.39$ E$-4)$. {\it Top}: $\alpha_1$ from
(\ref{f64}), {\it Down}: $\alpha_1 = 1$.}  \label{tab2}
\end{table}

\begin{table}[ht]
\begin{tabular}{ccccccc}
\hline\\[-0.52cm]
\hline \\[-0.35cm]
$\quad \quad {\rm Method} \quad \quad$ & $\quad n \quad$ & $\quad
\alpha_n \quad$ & $\quad d_n \quad$ & $\quad e_n \quad$  \\[0.1cm]
\hline \\[-0.35cm]
(TI/DP) & $\quad 6 \quad$ & $\quad 7.62$ E$-8\quad$ & $\quad 4.83$ E$-4\quad$ & $\quad 1.51$ E$-1\quad$ \\
(IIM/A1) & $\quad 5 \quad$ & $\quad 1.24$ E$-7\quad$ & $\quad 4.80$ E$-4\quad$ & $\quad 1.51$ E$-1\quad$ \\
(IIM/GS) & $\quad 9 \quad$ & $\quad 3.98$ E$-7\quad$ & $4.95$ E$-4$ & $\quad 1.62$ E$-1\quad$ \\[0.1cm]
\hline\\[-0.35cm]
(TI/DP) & $\quad 7 \quad$ & $\quad 7.62$ E$-8\quad$ & $\quad 4.83$ E$-4\quad$ & $\quad 1.51$ E$-1\quad$ \\
(IIM/A1) & $\quad 6 \quad$ & $\quad 1.24$ E$-7\quad$ & $\quad 4.80$ E$-4\quad$ & $\quad 1.51$ E$-1\quad$ \\
(IIM/GS) & $\quad 22 \quad$ & $\quad 4.76$ E$-7\quad$ & $5.00$ E$-4$ & $\quad 1.65$ E$-1\quad$ \\[0.1cm]
\hline\\[-0.52cm]
\hline
\end{tabular}
\caption{Example \ref{e61} (iii) with $\sigma = 0.01$ $(\delta =
\sigma \| y\|_2 \approx 4.60$ E$-4)$. {\it Top}: $\alpha_1$ from
(\ref{f64}), {\it Down}: $\alpha_1 = 1$.} \label{tab3}
\end{table}

In our numerical experiments we observed that the accuracy of each
individual regularization method in the three test cases of Examples
\ref{e61} (i) -- (iii) is as predicted by the theory. In Tables
\ref{tab1} -- \ref{tab3} we mainly concentrate on the performance of
the three methods and observe following:
\begin{enumerate}
\item As far as computational expenses are concerned, the iteration
method (IIM/A1) performs best. In fact, this method requires the
smallest number of iterations compared with the other two methods.
\item For the method (IIM/GS), the number of iterations can
considerably be reduced by starting with $\alpha_1$ from (\ref{f64})
instead of starting with $\alpha_1 = 1$. For the other two methods
(TI/DP) and (IIM/A1), the number of iterations differs only by 1 for
the two starting values (\ref{f64}) and $\alpha_1 = 1$,
respectively.
\item In all three iteration methods, the $\alpha$--sequence $(\alpha_k)_{1}^n$
is decreasing. However, the final regularization parameter
$\alpha_n$ is smallest for method (TI/DP). Comparing the
discrepancies $d_k$ for the individual iterations $k=1,2,...$ (which
are not contained in the tables) we observed that, for $k \ge 2$,
$d_k$ in method (IIM/A1) is always smaller than $d_k$ in method
(TI/DP).

\end{enumerate}

\end{document}